\documentclass[onefignum,onetabnum]{siamonline190516}

\usepackage{lipsum}
\usepackage{amsfonts}
\usepackage{graphicx}
\usepackage{epstopdf}
\usepackage{algorithmic}
\ifpdf
  \DeclareGraphicsExtensions{.eps,.pdf,.png,.jpg}
\else
  \DeclareGraphicsExtensions{.eps}
\fi

\usepackage{verbatim}
\usepackage[format=plain, labelfont=bf, singlelinecheck=off]{caption} 
\usepackage{booktabs}
\usepackage{array}
\usepackage{multirow}

\newcommand{\bfb}{\mathbf{b}}

\newcommand{\bfh}{\mathbf{h}}

\newcommand{\bfk}{\mathbf{k}}
\newcommand{\bff}{\mathbf{f}}

\newcommand{\bfr}{\mathbf{r}}
\newcommand{\bfs}{\mathbf{s}}
\newcommand{\bfu}{\mathbf{u}}
\newcommand{\bfv}{\mathbf{v}}

\newcommand{\bfx}{\mathbf{x}}
\newcommand{\bfy}{\mathbf{y}}
\newcommand{\bfz}{\mathbf{z}}

\newcommand{\bfA}{\mathbf{A}}
\newcommand{\bfB}{\mathbf{B}}

\newcommand{\bfD}{\mathbf{D}}

\newcommand{\bfH}{\mathbf{H}}

\newcommand{\bfP}{\mathbf{P}}
\newcommand{\bfQ}{\mathbf{Q}}
\newcommand{\bfR}{\mathbf{R}}

\newcommand{\bfU}{\mathbf{U}}
\newcommand{\bfV}{\mathbf{V}}
\newcommand{\bfW}{\mathbf{W}}
\newcommand{\bfX}{\mathbf{X}}

\newcommand{\bfbeta}{\boldsymbol \beta}

\newcommand{\bftheta}{\boldsymbol \theta}
\newcommand{\bfvartheta}{\boldsymbol \vartheta}

\newcommand{\bfphi}{\boldsymbol \phi}
\newcommand{\bfzeta}{\boldsymbol \zeta}
\newcommand{\bfgamma}{\boldsymbol \gamma}
\newcommand{\bfsigma}{\boldsymbol \sigma}

\newcommand{\calX}{{\cal X}}

\usepackage{enumitem}
\setlist[enumerate]{leftmargin=.5in}
\setlist[itemize]{leftmargin=.5in}

\usepackage[english]{babel}
\newsiamremark{remark}{Remark}
\newsiamremark{assumption}{Assumption}

\newsiamremark{hypothesis}{Hypothesis}
\crefname{hypothesis}{Hypothesis}{Hypotheses}
\newsiamthm{claim}{Claim}

\headers{Objective Bayesian Analysis of a Cokriging Model}{Pulong Ma}

\title{Objective Bayesian Analysis of a Cokriging Model for Hierarchical Multifidelity Codes\thanks{{This material is supported by the U.S. NSF under Grant DMS-1638521 to the Statistical and Applied Mathematical Sciences Institute.}}}

\author{Pulong Ma\thanks{Statistical and Applied Mathematical Sciences Institute and Duke University, 4051 Research Commons, Suite 300, 79 T.W. Alexander Drive, P.O. Box 110207, Durham, NC 27709, USA.
  (\email{pulong.ma@duke.edu}, \url{https://pulongma.github.io}).}
}

\usepackage{amsopn}

\usepackage{mathrsfs}  

\ifpdf
\hypersetup{
  pdftitle={Objective Bayesian Analysis of a Cokriging Model for Hierarchical Multifidelity Codes},
  pdfauthor={Pulong Ma}
}
\fi

\begin{document}

\maketitle

\begin{abstract}
Autoregressive cokriging models have been widely used to emulate multiple computer models with different levels of fidelity. The dependence structures are modeled via Gaussian processes at each level of fidelity, where covariance structures are often parameterized up to a few parameters. The predictive distributions typically require intensive Monte Carlo approximations in previous works. This article derives new closed-form formulas to compute the means and variances of predictive distributions in autoregressive cokriging models that only depend on correlation parameters. For parameter estimation, we consider objective Bayesian analysis of such autoregressive cokriging models. We show that common choices of prior distributions, such as the constant prior and inverse correlation prior, typically lead to improper posteriors.  We also develop several objective priors such as the independent reference prior and the independent Jeffreys prior that are shown to yield proper posterior distributions. This development is illustrated with a borehole function in an eight-dimensional input space and applied to an engineering application in a six-dimensional input space. \end{abstract}



\section{Introduction}
Complex computer codes have been widely used to solve mathematical models that represent real-world processes in virtually every field of science and engineering. They are often referred to as \emph{simulators} in Uncertainty Quantification (UQ) and computer experiments \cite{OHagan2006, Santner2018}. In practice, computer codes can be too timing-consuming to be used for adequately addressing UQ tasks. To overcome this bottleneck, Gaussian processes have been widely used as surrogate models to approximate simulators due to its computational advantages and attractive theoretical properties \cite{Sacks1989}.     

In real applications, computer codes can be run at different levels of accuracy due to sophistication of physics incorporated in mathematical models, accuracy of numerical solvers and resolutions of meshes; see \cite{Peherstorfer2018} for formal definition of multifidelity models. Several works have been proposed to combine output from computer codes at different fidelity levels based on a well-known geostatistical method called \emph{cokriging}; see Chapter 3 of \cite{Cressie1993}. The approach of cokriging to synthesizing multiple computer model outputs is originated in \cite{Kennedy2000}, which is developed based upon a first order Markov assumption that given output from a low-fidelity code run at an input, no more information can be learnt about the high-fidelity code with output from the low-fidelity code at any other input. The resulting cokriging model is often referred to as an \emph{autoregressive cokriging} model.  Several extensions of this autoregressive cokriging model have been proposed with increased model flexibility and Bayesian inference approaches \cite{Qian2008, Gratiet2013, Gratiet2014}. 

A common feature found in these works is that the predictive distribution for the high-fidelity code given both low and high fidelity output at a set of inputs as well as correlation parameters (or range parameters) requires numerical integration, which leads to intensive computations. Indeed, the predictive distribution can be available in a closed-form when other model parameters such as regression parameters and variance parameters are conditioned upon, but this leads to under-estimation of uncertainties associated with predictions. The uncertainty analysis about the quantity of interest, often a transformation of predictors, will hence suffer severely from this artifact. To avoid this, we derive a new recursive predictive formula so that predictive distributions are only conditioned upon the code output and correlation parameters. Closed-form predictive means and predictive variances are also derived at each code level and they can be computed without the need of Monte Carlo approximations unlike those in \cite{Qian2008, Gratiet2013}. These closed-form formulas have the capability to explicitly account for the uncertainty due to the estimation of regression parameters, scale discrepancy parameter and variance parameters. 

Inference about model parameters has been approached in several different ways in autoregressive cokriging models with focus on empirical Bayesian approaches. A particular challenge in autoregressive cokriging models is to estimate correlation parameters. To tackle this issue, \cite{Kennedy2000} assume independent noninformative priors for all the parameters, and then carry out numerical maximization for the marginal likelihood functions after integrating out regression parameters with respect to scale discrepancy parameters, variance parameters and correlation parameters.  In \cite{Qian2008}, conjugate priors are assumed for regression parameters, scale discrepancy parameters and variance parameters. For correlation parameters, proper gamma priors are assumed. The correlation parameters are then estimated by maximizing the corresponding marginal posteriors. For the code at the first level, standard nonlinear optimization is performed; while for the code at the second level, the corresponding posterior does not have a closed form and its evaluation requires numerical integration. This optimization procedure is then recast into a stochastic programming problem. To alleviate computational difficulties in \cite{Kennedy2000, Qian2008}, \cite{Gratiet2013} develop an efficient joint Bayesian estimation approach with either non-information priors or informative priors for all the model parameters except correlation parameters. Without further assuming prior distributions for correlation parameters, \cite{Gratiet2013} maximizes a concentrated restricted likelihood to obtain estimates of correlation parameters at each code level. We refer to this estimation approach as \emph{plug-in MLE} hereafter. However, whether the choices of priors in \cite{Qian2008, Gratiet2013} will lead to good estimates is not discussed. We show that vague proper priors for correlation parameters in \cite{Qian2008} lead to an improper posterior. Thus, the usage of such vague priors will not solve but hide the problem; see \cite{Berger2006} for detailed discussion and references therein. We also show that the concentrated restricted likelihood in \cite{Gratiet2013} with noninformative priors and informative priors (when chosen to be vague) can have nonrobust estimates in autoregressive cokriging models, where nonrobustness is defined to be the situation where the correlation matrix becomes either singular or near-diagonal when correlation parameters go to infinity or zero; see \cite{Gu2018} for detailed discussions. 

This article has two primary objectives that are of interest from computational and theoretical perspectives. The first objective is the derivation of new formulas for predictive distributions of the code output at any level over a new input given code output and correlation parameters. We show that realizations from predictive distributions can be simulated based upon a set of conditional distributions. The predictive means and predictive variances can be computed exactly at any fidelity level in a computationally efficient way. The new closed-form predictive formulas explicitly take into account uncertainties due to the estimation of the location and scale parameters. The second objective is the development of objective priors. The objective priors can be used as default priors when elicitation of prior information is challenging. It also enables more accurate uncertainty estimation in the predictive distribution than the typical maximum-likelihood based approaches with commonly-used noninformative priors.

The rest of this paper is organized as follows. Section~\ref{sec: univariate model} reviews the autoregressive cokriging models for analyzing multifidelity codes. In Section~\ref{sec: prediction in univariate model}, new closed-form expressions are derived for predictive distributions at all code levels conditioned on all code output and correlation parameters. The $s$-level cokriging model turns out to have the same computational cost as $s$ independent kriging models for both parameter estimation and prediction. Section~\ref{sec: objective Bayes} begins with discussions on commonly-chosen noninformative priors and proves that the resulting posteriors are improper with such noninformative priors. The objective priors including independent reference priors and independent Jeffreys priors are then developed and are shown to yield proper posteriors. Section~\ref{sec: numerical demonstration} gives several numerical examples to demonstrate the new development. Section~\ref{sec: discussion} is concluded with further discussions.  

\section{The Autoregressive Cokriging Model} \label{sec: univariate model}

Suppose that we have $s$ levels of code $y_{1}(\cdot),\ldots,$ $y_{s}(\cdot)$, where the code $y_{t}(\cdot)$ is assumed to be more accurate than the code $y_{t-1}(\cdot)$ for $t=2,\ldots,s$. Let $\calX$ be a compact subset of $\mathbb{R}^{d}$, which is assumed to be the input space of computer code. Further assume that the code $y_{t}(\cdot)$ is run at a set of input values denoted by $\calX_{t}\subset\calX$ for $t=1,\ldots,s$, where $\calX_{t}$ is assumed to contain $n_{t}$ input values. Consider the following autoregressive model as in \cite{Kennedy2000, Gratiet2013}:
\begin{align} \label{eqn: AR}
y_{t}(\bfx)=\gamma_{t-1}y_{t-1}(\bfx)+\delta_{t}(\bfx), \, \bfx  \in \calX,
\end{align}
for $t=2,\ldots,s$, where $y_{t-1}(\cdot)$ is an unknown function of input. $\delta_{t}(\cdot)$ is the unknown location discrepancy function representing the local adjustment from level $t-1$ to level $t$. $\gamma_{t-1}$ is the scale discrepancy representing the scale change from level $t-1$ to level $t$. Notice that currently $\gamma_{t-1}$ does not depend on input. A more general assumption is to take $\gamma_{t-1}(\cdot)$ to be a basis-function representation, i.e., $\gamma_{t-1}(\cdot)=\bfk_{t-1}(\cdot)^{\top}\bfzeta_{t-1}$ for $t=2,...,s$, where $\bfk_{t-1}(\cdot)$ is a vector of basis functions and $\bfzeta_{t-1}$ is a vector of unknown coefficients with dimension $q_{\zeta}$. The development in this article is true for this general parameterization. Without loss of generality, we focus on the simple form, i.e., $\bfk_{t-1}$ is assumed to be  1 and $\bfzeta_{t-1}$ is assumed to be a scalar parameter.   

To account for uncertainties in the unknown functions $y_{1}(\cdot)$
and $\delta_{t}(\cdot)$, Gaussian process priors can be assigned:
\begin{align}
\begin{split}y_{1}(\cdot)\mid\bfbeta_{1},\sigma_{1}^{2},\bfphi_{1} & \sim\mathcal{GP}(\bfh_{1}(\cdot)^{\top}\bfbeta_{1},\,\sigma_{1}^{2}r(\cdot,\cdot|\bfphi_{1})),\\
\delta_{t}(\cdot) & \sim\mathcal{GP}(\bfh_{t}(\cdot)^{\top}\bfbeta_{t},\,\sigma_{t}^{2}r(\cdot,\cdot|\bfphi_{t})),
\end{split}
\label{eqn: co-kriging model}
\end{align}
for $t=2,\ldots,s$, $i=1,\ldots,n_{t}$, where $r(\cdot,\cdot|\bfphi_{t})$ is a correlation function with correlation parameters $\bfphi_{t}$. A popular choice is to choose the product form of correlations with the power-exponential family and the Mat\'ern family.  $\bfh_{t}(\cdot)$ is a vector of (fixed) basis functions and $\bfbeta_{t}$ is a vector of unknown coefficients at code level $t$. $\sigma^2_t$ is the variance parameter.

The cokriging model defined by~\eqref{eqn: AR} and~\eqref{eqn: co-kriging model} has been used to model computer model output at different fidelity levels in previous works \cite{Kennedy2000, Gratiet2013} with hierarchically nested design, i.e., $\calX_{t}\subset\calX_{t-1}$. In \cite{Qian2008}, a measurement-error process is incorporated in \eqref{eqn: AR} to link field observations with computer model outputs at two fidelity levels. In these works, the assumption of hierarchically nested designs is imposed in order to allow for closed-form likelihood-based inference. Let ${\bfy}_{t}$
be a vector of output values at all inputs in ${\calX}_{t}$ at code level $t$. Let $\bfbeta:=(\bfbeta_1^\top, \ldots, \bfbeta_s^\top)^\top$, $\bfgamma:=(\gamma_1, \ldots, \gamma_{s-1})^\top$, $\bfsigma^2:=(\sigma^2_1, \ldots, \sigma^2_s)^\top$, $\bfphi:=(\bfphi_1^\top, \ldots, \bfphi_s^\top)^\top$,
and $\bfy=(\bfy_{1}^{\top},\ldots,\bfy_{s}^{\top})^{\top}$. Then the marginal likelihood is 
\begin{align} \label{eqn: sampling dist}
\begin{split}L({\bfy}\mid\bfbeta,\bfgamma,\bfsigma^{2},\bfphi) & =\pi({\bfy}_{1}\mid\bfbeta_{1},\sigma_{1}^{2},\bfphi_{1})\prod_{t=2}^{s}\pi({\bfy}_{t}\mid{\bfy}_{t-1},\gamma_{t-1},\bfbeta_{t},\sigma_{t}^{2},\bfphi_{t})\end{split},
\end{align}
where  
\begin{align}
\begin{split}\pi({\bfy}_{1}\mid\bfbeta_{1},\sigma_{1}^{2},\bfphi_{1}) & =\mathcal{N}({\bfH}_{1}\bfbeta_{1},\,\sigma_{1}^{2}{\bfR}_1),\\
\pi({\bfy}_{t}\mid {\bfy}_{t-1},\gamma_{t-1},\bfbeta_{t},\sigma_{t}^{2},\bfphi_{t}) & =\mathcal{N}({\bfH}_{t}\bfbeta_{t} + {W}_{t-1}\gamma_{t-1},\,\sigma_{t}^{2} {\bfR}_t),
\end{split}
\label{eqn: conditional sampling distributions}
\end{align}
with ${\bfH}_t:=\bfh_t({\calX}_t)$, $\bfR_t:=r(\calX_t, \calX_t\mid \phi_t)$, and ${W}_{t-1}:=y_{t-1}({\calX}_t)$, where $y_{t-1}(A):=[y_{t-1}(x),x\in A]$ is a vector of output values over inputs in $A$. This sampling distribution provides a convenient form to perform closed-form likelihood-based inference. 

\section{The Cokriging Predictor and Cokriging Variance} \label{sec: prediction in univariate model}

For any new input $\bfx_{0}\in \calX$, the goal is to make prediction for $y_s(\bfx_0)$ based upon the code output $\bfy$. In \cite{Kennedy2000}, a closed-form predictive distribution is derived for $\pi(y_s(\bfx_0) \mid \bfy, \bfgamma, \bfsigma^2, \bfphi)$, which only accounts for the uncertainty due to estimation of $\bfbeta$ and has $O((\sum_{t=1} n_t)^3)$ computational cost. In \cite{Qian2008}, a closed-form predictive distribution is only given for $\pi(y_s(\bfx_0) \mid \bfy, \bfbeta, \bfgamma, \bfsigma^2, \bfphi)$, which also has $O((\sum_{t=1}^s n_t)^3)$ computational cost. To account for uncertainty due to estimation of model parameters $\bfbeta, \bfgamma, \bfsigma^2, \bfphi$, Monte Carlo approximation is required. In \cite{Gratiet2013}, a closed-form predictive distribution is also only given for $\pi(y_s(\bfx_0) \mid \bfy, \bfbeta, \bfgamma, \bfsigma^2, \bfphi)$, and Monte Carlo approximation is used to account for uncertainty due to estimation of the model parameters. \cite{Gratiet2013} also develops an iterative formula to invert the $(\sum_{t=1}^s n_t)\times (\sum_{t=1}^s n_t)$ correlation matrix of code output at all levels, which reduces computation cost to $O(\sum_{t=1}^s n_t^3)$. 

In what follows, we give a new way to derive closed-form predictive distributions for $\pi(y_t(\bfs_0) \mid \bfy, \bfphi), t=1, \ldots, s$ that not only explicitly account for the uncertainty due to estimation of $\bfbeta, \bfgamma, \bfsigma^2$ but also has $O(\sum_{t=1}^s n_t^3)$ computational cost. The formula for these predictive distributions is derived based upon the idea that the new input $\bfx_0$ is added to each ${\calX}_t$ such that a hierarchically nested design can be obtained.

To deal with these unknown parameters $\bfbeta, \bfgamma, \bfsigma^2$, the following standard reference priors are used for the location-scale parameters:
$\bfbeta,\bfgamma,\bfsigma^{2}$: 
\begin{align} \label{eqn: location-scale prior}
\begin{split}
\pi^{R}(\bfbeta_{1},\sigma_{1}^{2}) & \propto\frac{1}{\sigma_{1}^{2}},\\
\pi^{R}(\bfbeta_{t},\gamma_{t-1},\sigma_{t}^{2}) & \propto\frac{1}{\sigma_{t}^{2}},\,t=2,\ldots,s.
\end{split}
\end{align}
 
The following lemma gives the predictive distribution of $\bfy(\bfx_0):=(y_1(\bfx_0), \ldots, y_s(\bfx_0))^\top$ given $\bfy$ and $\bfphi$.

\begin{lemma} \label{lem: conditional predictive dists}
According to the cokriging model defined in \eqref{eqn: AR} and \eqref{eqn: co-kriging model}, the joint predictive distribution of $\bfy(\bfx_0)$ given all the code output $\bfy$ and correlation parameters $\bfphi$ is 
\begin{align}\label{eqn: joint predictive dist under noninformative priors}
\begin{split}
\pi(\bfy(\bfx_0)\mid \bfy,\bfphi) &=\pi(y_1(\bfx_0)\mid \bfy_{1},\bfphi_{1})\prod_{t=2}^{s-1}\pi(y_t(\bfx_0)\mid {\bfy}_{t-1}, y_{t-1}(\bfx_0), \bfy_{t},\bfphi_{t})\\
&\quad \times \pi(y_s(\bfx_0) \mid y_{s-1}(\bfx_0), \bfy_s, \bfphi_s).
\end{split}
\end{align}
Here conditional distributions on the right-hand side
are Student $t$-distributions $t_{n_t-q_t}(\mu_t(\bfx_0), $ $ \Sigma_t(\bfx_0))$ given by
\begin{align*}
\begin{split}
\mu_{t}(\bfx_0) & :=\bfX_t^\top(\bfx_0)\hat{\bfb}_t+\bfr_t^\top(\bfx_0)\bfR_t^{-1}(\bfy_{t}-\bfX_t\hat{\bfb}_t),\\
\Sigma_{t}(\bfx_0) & :=\hat{\sigma}^2_t c_t^*,
\end{split}
\end{align*}
with 
\begin{align*}
\begin{split}
\hat{\sigma}^2_t & :=(\bfy_{t}-\bfX_t\hat{\bfb}_t)^{\top}\bfR_t^{-1}(\bfy_{t}-\bfX_t\hat{\bfb}_t) / (n_t-q_t),\\
c_t^* & :=r(\bfx_0,\bfx_0|\bfphi_{t})-\bfr_t^\top(\bfx_0)\bfR_t^{-1}\bfr_t(\bfx_0) \\
&\quad + [\bfX_{t}(\bfx_0)-\bfX_{t}^{\top}\bfR_{t}^{-1}\bfr_t(\bfx_0)]^{\top}(\bfX_{t}^{\top}\bfR_{t}^{-1}\bfX_{t})^{-1}[\bfX_{t}(\bfx_0)-\bfX_{t}^{\top}\bfR_{t}^{-1}\bfr_t(\bfx_0)],
\end{split}
\end{align*}
where $\hat{\bfb}_t:=(\hat{\bfbeta}^\top_t, \hat{\gamma}_{t-1})^\top = (\bfX_t^\top \bfR_t^{-1} \bfX_t)^{-1}\bfX_t^\top \bfR_t^{-1} \bfy_t$.  $\bfr_t(\bfx_0):=r({\calX}_t, \bfx_0\mid \bfphi_t)$. $\bfX_1:=\bfH_1$.  $\bfX_t:=[{\bfH}_{t}, y_{t-1}(\calX_t)]$ for $t>1$. $\bfX_1(\bfx_0):=\bfh_1(\bfx_0)$ and $\bfX_t(\bfx_0):=[\bfh_t^\top(\bfx_0), y_{t-1}(\bfx_0)]^\top$ for $t>1$. $q_t$ is the number of columns in $\bfX_t$. Notice that $\hat{\bfbeta}_t$ is the generalized least squares estimate of $\bfbeta_t$ and $\hat{\gamma}_{t-1}$ is the generalized least squares estimate of $\gamma_{t-1}$.
\end{lemma}

\begin{proof}
See Appendix~\ref{app: conditional predictive dist}.
\end{proof}

This result shows that one can generate a random sample from the predictive distribution $\pi(\bfy(\bfx_0) \mid \bfy, \bfphi)$ by sequentially sampling from a collection of conditional distributions. Notice that samples are obtained across all the code levels. The computation associated with each conditional distribution only requires $O(n_t^3)$ flops for $t=1, \ldots, s$. Before stating the next theorem that provides a convenient way to exactly compute the predictive mean and predictive variance in the predictive distribution $\pi(\bfy(\bfx_0) \mid \bfy, \bfphi)$, we define the cokriging predictor and cokriging variance at each code level. 

\begin{definition}
Let $\hat{\bfy}(\bfx_0)=(\hat{y}_1(\bfx_0), \ldots, \hat{y}_s(\bfx_0))^\top$ be a vector of predictive means with $\hat{y}_t(\bfx_0):=E[y_t(\bfx_0)\mid \bfy, \bfphi]$ and $\hat{\bfv}(\bfx_0)=(\hat{v}_1(\bfx_0),\ldots, \hat{v}_s(\bfx_0))^\top$ be a vector of predictive variances with $\hat{v}_t(\bfx_0):=Var[y_t(\bfx_0)\mid \bfy, \bfphi]$. In what follows, $\hat{\bfy}(\bfx_0)$ is called the cokriging predictor and $\hat{\bfv}(\bfx_0)$ is called the cokriging variance for all levels of code at new input $\bfx_0$.
\end{definition}

\begin{theorem} \label{thm: predictive mean and variance}
Suppose that $n_t-q_t>2$ such that the $t$-distributions in \eqref{eqn: joint predictive dist under noninformative priors} have valid variances. Then the cokriging predictor and cokriging variance at code level $t$ are given by 
\begin{align} \label{eqn: cokriging predictor}
    \begin{split}
    \hat{y}_t(\bfx_0) &= \bff_t^\top(\bfx_0) \hat{\bfb}_t + \bfr_t^\top(\bfx_0) \bfR_t^{-1}(\bfy_t - \bfX_t \hat{\bfb}_t), \\
    \hat{v}_t(\bfx_0) &= \hat{\gamma}_{t-1}^2\hat{v}_{t-1}(\bfx_0) + \frac{n_t-q_t}{n_t-q_t-2} \hat{\sigma}^2_t \left\{r(\bfx_0, \bfx_0|\bfphi_t) - \bfr_t^\top(\bfx_0) \bfR_t^{-1} \bfr_t(\bfx_0) + \kappa_t \right\},
    \end{split}
\end{align}
where $\bff_1(\bfx_0):=\bfh_1(\bfx_0)$, $\bff_t(\bfx_0):=[\bfh_t^\top(\bfx_0), \hat{y}_{t-1}^\top(\bfx_0)]^\top$ for $t>1$, $\hat{v}_0:=0$, and 
\begin{align*}
    \kappa_t &:=  [\bff_{t}(\bfx_0)-\bfX_{t}^{\top}\bfR_{t}^{-1}\bfr_t(\bfx_0)]^{\top}(\bfX_{t}^{\top}\bfR_{t}^{-1}\bfX_{t})^{-1}[\bff_{t}(\bfx_0)-\bfX_{t}^{\top}\bfR_{t}^{-1}\bfr_t(\bfx_0)] \\
    &\quad + \hat{v}_{t-1}(\bfx_0) \left\{ y_{t-1}^\top(\calX_t) \bfQ^H_t y_{t-1}(\calX_t) \right\}^{-1}.
\end{align*}
with $\bfQ^H_t := \bfR_t^{-1} - \bfR_t^{-1} \bfH_t (\bfH_t^\top \bfR_t^{-1} \bfH_t)^{-1} \bfH_t^\top \bfR_t^{-1}$.
\end{theorem}
\begin{proof}
See Appendix~\ref{app: predictive mean and variance}.
\end{proof}

Theorem~\ref{thm: predictive mean and variance} shows that the predictive mean and predictive variance in~\eqref{eqn: cokriging predictor} can be computed exactly with computational cost $O(\sum_{t=1}^s n_t^3)$. As a byproduct, predictions at code levels from $t=1$ to $t=s-1$ are obtained automatically.  For $t=1$, the predictive mean and predictive variance in the autoregressive cokriging model are exactly the universal kriging predictor and universal kriging variance in a kriging model. For $t>1$, the cokriging predictor is a sum of a kriging predictor and an additional constant, i.e., $\hat{y}_t(\bfx_0)=\bfh_t^\top(\bfx_0)\hat{\bfbeta}_t + \bfr_t^\top(\bfx_0)\bfR_t^{-1}(\bfy_t - \bfH_t\hat{\bfbeta}_t) + [\hat{y}_{t-1}(\bfx_0) + \bfr_t^\top(\bfx_0)\bfR_t^{-1}W_{t-1}]\hat{\gamma}_{t-1}$. The computational cost for both parameter estimation and prediction in an $s$-level autoregressive cokriging model is equivalent to the one in $s$ independent kriging models. This predictive distribution allows us to explicitly integrate out models parameters $\bfbeta, \bfgamma, \bfsigma^2$ except the range parameters $\bfphi$, and hence carries several advantages over the predictive distribution given in \cite{Gratiet2013}. Specifically, the predictive distribution in \cite{Gratiet2013} is a normal distribution when all model parameters $\{\bfbeta, \bfgamma, \bfsigma^2, \bfphi\}$ are conditioned upon. Thus, intensive computation of Monte Carlo approximations is required to account for the uncertainty due to estimation of $\{\bfbeta, \bfgamma, \bfsigma^2\}$ in order to derive the predictive distribution $\pi(y_s(\bfx_0)\mid \bfy, \bfphi)$.

Theorem~\ref{thm: predictive mean and variance} is the first result that gives the recursive formula for predictive distributions in autoregressive cokriging models where only the correlation parameters are conditioned upon. Our proposed recursive formula differs from the one developed by \cite{Gratiet2014} in the following aspect. \cite{Gratiet2014} use a cokriging model that is different from what is presented in this paper. In particular, \cite{Gratiet2014} represent the higher-fidelity code output $y_t(\cdot)$ as a transformation of the conditional distribution of $y_{t-1}(\cdot)$ given the code output $\bfy_{t-1}^{\mathscr{D}} :=\{ \bfy_1, \ldots, \bfy_{t-1} \}$ and model parameters $\{\bfbeta_{t-1}, \gamma_{t-2}, \sigma^2_{t-1}, \bfphi_{t-1} \}$. In what follows, the cokriging model in \cite{Gratiet2014} will be referred to as the \emph{recursive cokriging} model. Essentially, this recursive cokriging model aims at modeling the \emph{conditional distribution} $\pi(y_t(\cdot) \mid \bfy_{t}^{\mathscr{D}}, \bfbeta_{t}, \gamma_{t-1}, \sigma^2_{t}, \bfphi_{t})$ directly. Proposition 2 of \cite{Gratiet2014} shows that this recursive cokriging model allows the computation of the predictive distribution $\pi(y_s(\cdot) \mid \bfy_{s}^{\mathscr{D}}, \bfphi_{s})$ at the highest fidelity level in a recursive fashion with the same computational cost as the one in Theorem~\ref{thm: predictive mean and variance}. However, the recursive cokriging model does not give the predictive distributions $\pi(y_t(\cdot) \mid \bfy_{s}^{\mathscr{D}}, \bfphi_{s})$ at intermediate fidelity levels for $t=1, \ldots, s-1$, since it only provides the conditional distribution $\pi(y_{t}(\cdot) \mid \bfy_{t}^{\mathscr{D}}, \bfphi_{t})$ for $t=1, \ldots, s-1$ at intermediate fidelity levels and the predictive distribution $\pi(y_s(\cdot) \mid \bfy_{s}^{\mathscr{D}}, \bfphi_{s})$ at the highest fidelity level. It is clear that the set of conditional distributions is different from the set of predictive distributions at intermediate fidelity levels. In the design of experiments for multifidelity codes, it is often very useful to obtain both the predictive distribution at the highest fidelity level and the predictive distributions at intermediate fidelity levels.

The following corollary highlights the properties of autoregressive cokriging predictors in Theorem~\ref{thm: predictive mean and variance}.

\begin{corollary} \label{cor: cokriging and UK}
Let $\bfx_0$ be a new input in the domain $\mathcal{X}$. Let $\hat{y}_{t}^{K}(\bfx_0):=E\{y_t(\bfx_0) \mid \bfy_t, \bfphi_t\}$  and  $\hat{v}_t^{K}(\bfx_0): = Var\{ y_t(\bfx_0) \mid \bfy_t, \bfphi_t\}$ be the kriging predictor and kriging variance based on data $\{\bfy_{t}, \calX_{t}\}$ with fixed basis functions given by $\bfX_t$. 
    If $\bfx_0 \in \calX_{t} \setminus \calX_{t+1}$ and  with $t=1, \ldots, s-1$, we have $\hat{y}_{\ell}(\bfx_0) = \hat{y}_{\ell}^K(\bfx_0)=y_t(\bfx_0)$ and $\hat{v}_{\ell}(\bfx_0) = \hat{v}_{\ell}^K(\bfx_0) = 0$ for $\ell = 1, \ldots, t$.
    If $\bfx_0\notin \calX_t$, we have  $\hat{v}_t(\bfx_0)\geq \hat{v}_t^K(\bfx_0)$, where the equality holds when the condition that $\bfx_0\in \mathcal{X}_{t-1}$ is further imposed. If $\bfx_0 \notin \calX_{t-1}$, we have $\hat{v}_t(\bfx_0)> \hat{v}_t^K(\bfx_0)$.
\end{corollary}

{Corollary}~\ref{cor: cokriging and UK} indicates that cokriging predictors can be interpolators as kriging predictors. If the input $\bfx_0$ belongs to the design $\mathcal{X}_s$ in the highest fidelity code, the resulting predictive variances at $\bfx_0$ across all levels are zeros, i.e., $\hat{v}_t=0$ for $t=1, \ldots, s$. In other words, the cokriging predictors in~\eqref{eqn: cokriging predictor} are interpolators at all levels. When prediction is made at new inputs, the cokriging predictor has predictive variance no smaller than that associated with the kriging predictor. The extra uncertainty in cokriging variances comes from the uncertainty from lower code levels and the uncertainty to estimate the scale discrepancy parameter. However, without data coming from lower code levels, one cannot estimate the correlation parameters in the higher level very well, since in practice the higher level code is too expensive to get sufficient number of runs that can be used to obtain fairly good parameter estimates and hence prediction.

\section{Objective Bayesian Analysis} \label{sec: objective Bayes}
In Uncertainty Quantification, Bayesian analysis of cokriging models has been focused on using conjugate priors and noninformative priors \cite{Kennedy2000, Qian2008, Gratiet2013}. In this section, we focus on objective Bayesian analysis of the cokriging model, since objective priors can be used as default priors for Bayesian analysis \cite{Berger2006}, and have been often used in Gaussian process modeling \cite{Berger2001, Paulo2005, Gu2018}.

\subsection{Commonly-used Improper Priors}
A Gaussian process model has commonly-used priors; see \cite{Berger2001} for detailed discussions. Following this convention, the following priors will be referred to as {commonly-used} priors for parameters in autoregressive cokriging models. We consider the improper prior density for $\bftheta:=\{\bfbeta, \bfgamma, \bfsigma^2, \bfphi\} \in \Omega=\mathbb{R}^{sp} \times \mathbb{R}^s \times (0, \infty)^s\times (0, \infty)^{sd}$ of the form 
\begin{align}\label{eqn: general prior}
\begin{split}
    \pi(\bfbeta,\bfgamma,\bfsigma^{2},\bfphi)
    & \propto \frac{\pi(\bfphi)}{\prod_{t=1}^s(\sigma^2_t)^{a_t}}, a_t \in \mathbb{R},
    \end{split}
\end{align}
for various choices of $\pi(\bfphi)$ and $a_t$. Note that Kennedy and O'Hagan \cite{Kennedy2000} do not assume a prior for $\bfphi$, but the form of prior in \cite{Kennedy2000} is the same as the one in~\eqref{eqn: general prior} when inverse range priors are chosen for $\bfphi$ in \eqref{eqn: general prior}: $\pi(\bfphi) = \prod_{t=1}^s \prod_{\ell=1}^d \phi_{t,\ell}^{-1}$ and $a_t=1$. Then \cite{Kennedy2000} estimate the parameter $\bfphi$ by maximizing the distribution $p(\bfy \mid \bfgamma, \bfsigma, \bfphi)$. Similarly, Gratiet \cite{Gratiet2013} do not assume a prior for $\bfphi$ and estimate $\bfphi$ based on a concentrated restricted likelihood via a restricted maximum likelihood approach \cite{Harville1974}. In Section~\ref{sec: integrated likelihood}, we show that constant flat prior or inverse range prior for $\bfphi$ can lead to improper posteriors.

\subsection{Commonly-used Proper Priors}
As the commonly-used improper priors may lead to improper posteriors, one obvious way to guarantee propriety of the posterior distribution is to assume proper priors, assessed either subjectively or from previous data, however, for Gaussian processes, the correlation parameter $\bfphi$ can be difficult to interpret. Another way is to assume vague proper priors, however, this can only hide the problem when the posterior concentrates its mass at zero. If the posterior impropriety is occurring because the posterior is not decreasing at infinity, then the empirical Bayes estimates of $\bfphi$ can be bad. In Appendix~\ref{app: Qian}, we discuss that proper priors in \cite{Qian2008} lead to improper posterior when they are chosen to be vague. Another choice is to use conjugate priors for parameters $\{ \bfbeta, \bfgamma, \bfsigma^2\}$ and leave the prior for $\bfphi$ unspecified and work with a concentrated restricted likelihood as in \cite{Gratiet2013}. We will show that this concentrated restricted likelihood is not decreasing to zero and can be maximized either at zero or at infinity; see Appendix~\ref{app: Gratiet} for detailed discussions. The property of empirical Bayes estimates of correlation parameters has been studied in great details in \cite{Gu2018} for Gaussian processes. The conclusions in \cite{Gu2018} can be potentially generalized for autoregressive cokriging models. To overcome these problems, objective priors of the form~\eqref{eqn: general prior} are developed and they are shown to yield proper posteriors in autoregressive cokriging models. Although this article is not focused on robust estimation, the independent reference prior can lead to robust estimation with parameterization given in \cite{Gu2018}.

\subsection{Integrated Likelihood} \label{sec: integrated likelihood}
The objective Bayesian analysis for Gaussian processes heavily uses an \emph{integrated likelihood}; see \cite{Berger2001} for example. For the autoregressive cokriging model in Section~\ref{sec: univariate model}, it is also possible to calculate such integrated likelihood.  Indeed, the integration of the product of the marginal likelihood~\eqref{eqn: sampling dist} and the prior~\eqref{eqn: general prior} with respect to the prior over $(\bfbeta, \bfgamma, \bfsigma^2)$ yields 
\begin{align*}
\begin{split}
\int L({\bfy}\mid\bfbeta,\bfgamma,\bfsigma^{2},\bfphi)\pi(\bfbeta,\bfgamma,\bfsigma^{2},\bfphi)\,d(\bfbeta,\bfgamma,\bfsigma^{2})
= L^I(\bfphi\mid \bfy) \pi(\bfphi),
\end{split}
\end{align*}
with the {integrated likelihood} function of $\bfphi$, $L^I(\bfphi\mid \bfy)$,  given by 
\begin{align} \label{eqn: integrated likelihood}
\begin{split}
    L^I(\bfphi\mid \bfy) 
    & \propto \prod_{t=1}^s |{\bfR}_t|^{-1/2} |{\bfX}_t^\top {\bfR}_t^{-1} {\bfX}_t|^{-1/2} \{S^2(\bfphi_t)\}^{-(({n}_t - q_t)/2+a_t-1)}, 
\end{split}
\end{align}
where ${\bfX}_1:={\bfH}_1$ and ${\bfX}_t:=[{\bfH}_t, {W}_{t-1}]$ with ${W}_{t-1}:=y_{t-1}({\calX}_t)$ for $t=2, \ldots, s$. $S^2(\bfphi_t):=\bfy_{t}^\top \bfQ_t \bfy_{t}$ with $\bfQ_t:=\bfR_t^{-1}\bfP$ and $\bfP:=\mathbf{I} - \bfX_t (\bfX_t^\top \bfR_t^{-1} \bfX_t)^{-1} \bfX_t^\top \bfR_t^{-1}$. The detailed computation to calculate the integrated likelihood is given in Appendix~\ref{app: integrated likelihood}. Notice that the posterior distribution of $\{\bfbeta, \bfgamma, \bfsigma^2, \bfphi\}$ is given by the product of the quantity $\int_{(0, \infty)^{sd}} L^I(\bfphi\mid \bfy) \pi(\bfphi) d \bfphi$ and a normalizing constant that does not depend on any model parameters $\{\bfbeta, \bfgamma, \bfsigma^2, \bfphi\}$. Thus, the posterior distribution of $\{\bfbeta, \bfgamma, \bfsigma^2, \bfphi\}$ is proper if and only if 
\begin{align} \label{eqn: condition for properiety}
    0< \int_{(0, \infty)^{sd}} L^I(\bfphi\mid \bfy) \pi(\bfphi) d \bfphi < \infty. 
\end{align}

To study the behavior of the integrated likelihood at zero and at infinity, the following mild assumptions are needed. 
\begin{assumption} \label{ass: correlation}
Given $1\leq t\leq s$ and $1\leq \ell \leq d$, suppose that $K(u)$ is a continuous function of $\phi_{t, \ell}>0$ for any $u>0$ such that:
\begin{itemize}[noitemsep,topsep=0pt]
\item[(i)] $K(u) = r(u/\phi_{t,\ell})$, where $r(\cdot)$  is a correlation function that satisfies $\lim_{u\to \infty} r(u) = 0$.
 \item[(ii)] As $\phi_{t,\ell} \to \infty$, the correlation matrix satisfies $\bfR_{t, \ell} = \mathbf{1} \mathbf{1}^\top + \nu_{t,\ell}(\phi_{t, \ell}) \bfD_{t, \ell} + \nu_{t,\ell}(\phi_{t, \ell}) $ $\omega_{t,\ell}(\phi_{t, \ell}) (\bfD^*_{t, \ell} + \bfB_{t,\ell}(\phi_{t,\ell})) $, where $\nu_{t, \ell}(\phi_{t, \ell})>0$ is a continuous function of $\phi_{t,\ell}$, $\bfD_{t,\ell}$ is a fixed nonsingular matrix with $\mathbf{1}^\top \bfD_{t, \ell}^{-1} \mathbf{1}\neq 0$ and $\mathbf{1}=(1, \ldots, 1)^\top$. $\bfD^*_{t, \ell} $ is a fixed matrix. $\bfB_{t,\ell}(\phi_{t,\ell})$ is a differentiable matrix satisfying 
 \begin{align*}
 \nu_{t,\ell}(\phi_{t, \ell}) \to 0, \quad \omega_{t, \ell}(\phi_{t,\ell}) \to 0, \quad \frac{\omega'_{t, \ell}(\phi_{t, \ell})}{ \frac{\partial \log \nu_{t, \ell}(\phi_{t, \ell}) }{\partial \phi_{t, \ell}} } \to 0, \\
 \| \bfB_{t,\ell}(\phi_{t,\ell}) \|_{\infty} \to 0, \quad \frac{\| \frac{\partial \bfB_{t,\ell}(\phi_{t,\ell})}{\partial \phi_{t,\ell}}  \|_{\infty}}{\frac{\partial}{\partial \phi_{t, \ell}} \log(\omega_{t, \ell}(\phi_{t, \ell}))} \to 0, 
 \end{align*}
 where $\| \bfB\|_{\infty} = \max_{i,j} |a_{i,j}|$ with $a_{i,j}$ being the $(i,j)$ entry of the matrix $\bfB$.
\end{itemize} 
\end{assumption}
These assumptions hold for all the correlation functions including power-exponential, spherical, rational quadratic, and Mat\'ern according to Table 1 in \cite{Gu2018}. The first assumption requires that the correlation function should decrease to zero as the distance between two points tends to infinity. The second assumption guarantees that the first two terms in the Taylor expansion of the correlation function decrease to zero as $\phi_{t, \ell} \to \infty$. However, an anonymous referee pointed out that the assumption of nonsingularity for the matrix $\bfD_{t,\ell}$ could be problematic for the Gaussian correlation and the Mat\'ern correlation with smoothness parameter greater than or equal to one according to \cite{Mure2018}. For this reason, this article will focus on the power-exponential correlation with roughness parameter less than 2 and the Mat\'ern correlation with the smoothness parameter less than one. For the Gaussian correlation and the Mat\'ern correlation with smoothness parameter greater than or equal to one, the objective priors can be developed according to the development in \cite{Mure2018}, but the technical details will be omitted in this article. 

The following lemma gives the behavior of the integrated likelihood at zero and at infinity.

\begin{lemma} \label{lem: behavior of integrated likelihood}
Note that $\bfphi_t=(\phi_{t,1}, \ldots, \phi_{t,d})^\top$. Let $\mathcal{C}(\bfX_t)$ be the column space of $\bfX_t$. For the cokriging model with sampling distribution~\eqref{eqn: sampling dist} and prior distribution~\eqref{eqn: general prior}, under mild conditions in Assumption~\ref{ass: correlation}, we have
\begin{enumerate}
    \item[(i)] If $\exists \ell$ such that $\phi_{t,\ell} \to 0^+$ for at least one $t$, the integrated likelihood exists and is greater than zero; that is, $\lim_{\phi_{t,\ell} \to 0^+} L^I(\bfphi \mid \bfy) >0$. 
    \item[(ii)] If $\phi_{t,\ell} \to \infty$ for all $\ell$ and $t$, the integrated likelihood satisfies 
    \begin{align*}
        L^I(\bfphi\mid \bfy) =             O\left( \prod_{t=1}^s \left(\sum_{\ell=1}^d \nu_{t,\ell}(\phi_{t,\ell} )\right)^{a_t-c_t} \right),
    \end{align*}
where $c_t:= 1_{\{\mathbf{1}\in \mathcal{C}(\bfX_t)\}} + \frac{1}{2} 1_{\{\mathbf{1}\notin \mathcal{C}(\bfX_t)\}}$ with $1_{\{\mathbf{1}\in \mathcal{C}(\bfX_t)\}}$ being 1 if the vector $\mathbf{1}$ is in $\mathcal{C}(\bfX_t)$ and zero otherwise.
\end{enumerate}
\end{lemma}
\begin{proof}
See Appendix~\ref{app: nu expression}.
\end{proof}
Lemma~\ref{lem: behavior of integrated likelihood} shows that both the flat prior $\pi(\bfphi_t) \propto 1$ and the noninformative prior $\pi(\bfphi_t)\propto \prod_{\ell=1}^d \phi_{t,\ell}^{-1}$ with $a_t=1$ in \eqref{eqn: general prior} lead to improper posteriors with the condition~\eqref{eqn: condition for properiety} violated. 

\subsection{Objective Priors} 
The posterior can be improper under certain common choices of priors in \eqref{eqn: general prior}. In what follows, several objective priors are derived and they are shown to yield proper posteriors.

Following \cite{Berger2001}, the parameters of interest are chosen to be $(\bfsigma^2, \bfphi)$ and $(\bfbeta, \bfgamma)$ are treated as the nuisance parameters. This specification leads to the prior factorization $\pi^R(\bftheta)=\pi^R(\bfbeta, \bfgamma\mid \bfsigma^2, \bfphi)\pi^R(\bfsigma^2, \bfphi)$. The Jeffreys-rule prior $\pi^R(\bfbeta, \bfgamma \mid \bfsigma^2, \bfphi)\propto 1$  is considered for the location parameters $(\bfbeta, \bfgamma)$ when other parameters are assumed known. Then the reference prior $\pi^R(\bfsigma^2, \bfphi)$ is computed based on the integrated likelihood with respect to $\pi^R(\bfbeta, \bfgamma) \propto 1$. Standard calculation yields the integrated likelihood $L^I(\bfsigma^2, \bfphi\mid \bfy)$:
\begin{align}
    \begin{split}
    L^I(\bfsigma^2, \bfphi\mid \bfy) &= \int_{\mathbb{R}^{q+1}} L(\bftheta \mid \bfy) \pi^R(\bfbeta, \bfgamma) d (\bfbeta, \bfgamma) \\
    & \propto \prod_{t=1}^s (\sigma^2_t)^{-(n_t-q_t)/2}|\bfR_t|^{-1/2}|\bfX_t^\top \bfR_t^{-1} \bfX_t|^{-1/2} \exp\left \{-\frac{S^2(\bfphi_t)}{2\sigma^2_t} \right\}.
    \end{split}
\end{align}

\begin{theorem}[\textbf{Independent Reference Prior}]\label{thm: reference prior}
Consider the group of parameters $\bftheta=(\bftheta_1, \ldots, \bftheta_s)$ with $\bftheta_t=(\bfbeta_t, \gamma_{t-1}, \sigma^2_t, \bfphi_t)$, where $\gamma_0:=0$. For the cokriging model with sampling distribution~\eqref{eqn: sampling dist}, the independent reference prior distribution, $\pi^R(\bftheta)$, is of the form~\eqref{eqn: general prior} with 
\begin{align}
    a_t = 1 & \text{ and } \pi^R(\bfphi) \propto \prod_{t=1}^s|I_t^R(\bfphi_t)|^{1/2},
\end{align}
where $I_t^R(\bfphi_t)$ is the Fisher information matrix by fixing all parameters except $\bftheta_t$:
\begin{align*}
I_t^R(\bfphi_t) = 
\begin{pmatrix}
n_t-q_t & \text{tr}(\bfW_{t,1})   & \text{tr}(\bfW_{t,2})      & \cdots & \text{tr}(\bfW_{t,d}) \\
    & \text{tr}(\bfW_{t,1}^2) & \text{tr}(\bfW_{t,1}\bfW_{t,2}) & \cdots & \text{tr}(\bfW_{t,1}\bfW_{t,d}) \\ 
    &                     &          & \ddots    & \vdots  \\
    &                     &         &   & \text{tr}(\bfW_{t,d}^2)
\end{pmatrix}_{(d+1)\times (d+1)},
\end{align*}
with $\bfW_{t,k}:=\dot{\bfR}^k_t \bfQ_t$ and $\dot{\bfR}^k_t: =\frac{\partial}{\partial \phi_{t,k}} \bfR_t$ being element-wise differentiation of $\bfR_t$ with respect to the parameter $\phi_{t,k}$ for $k=1, \ldots, d$.
\end{theorem}
\begin{proof}
See Appendix~\ref{app: reference prior}.
\end{proof}

\begin{theorem}[\textbf{Independent Jeffreys Priors}]\label{thm: Jeffreys prior}
Let $\bftheta=(\bftheta_1, \ldots, \bftheta_s)$ be the group of parameters with $\bftheta_t=(\bfbeta_t, \gamma_{t-1}, \sigma^2_t, \bfphi_t)$, where $\gamma_0:=0$. Suppose that $\bftheta_t$'s are independent. Then the independent Jeffreys prior, $\pi^{J1}$, obtained by assuming that $(\bfbeta_t, \gamma_t)$ and $(\sigma_t^2, \bfphi_t)$ are a priori independent, and the independent Jeffreys prior, $\pi^{J2}$, are of the form~\eqref{eqn: general prior} with 
\begin{align*}
    a_t = 1 & \text{ and } \pi^{J1}(\bfphi) \propto \prod_{t=1}^s|I_t^J(\bfphi_t)|^{1/2}, \\
    a_t = 1 + q_t/2 & \text{ and } \pi^{J2}(\bfphi) \propto \pi^{J1}(\bfphi) \prod_{t=1}^s |\bfX_t^\top \bfR^{-1}_t \bfX_t|^{1/2} ,
\end{align*}
where
\begin{align*}
I_t^J(\bfphi_t) = 
\begin{pmatrix}
n_t & \text{tr}(\bfU_{t,1})   & \text{tr}(\bfU_{t,2})      & \cdots & \text{tr}(\bfU_{t,d}) \\
    & \text{tr}(\bfU_{t,1}^2) & \text{tr}(\bfU_{t,1}\bfU_{t,2}) & \cdots & \text{tr}(\bfU_{t,1}\bfU_{t,d}) \\ 
    &                     &          & \ddots    & \vdots  \\
    &                     &         &   & \text{tr}(\bfU_{t,d}^2)
\end{pmatrix}_{(d+1)\times (d+1)},
\end{align*}
with $\bfU_{t,k}=\dot{\bfR}^k_t \bfR_t^{-1}, k=1, \ldots, d$, 
\end{theorem}
\begin{proof}
This result follows directly from Proposition 2.2 in \cite{Paulo2005}.
\end{proof}

\begin{theorem}
For the cokriging model with sampling distribution~\eqref{eqn: sampling dist}, the independent reference prior $\pi^R$, independent Jeffreys priors $\pi^{J1}, \pi^{J2}$ yield proper posteriors satisfying the condition~\eqref{eqn: condition for properiety}.
\end{theorem}
\begin{proof}
It follows from the results in \cite{Paulo2005} that 
\begin{align*}
    0<\int_{\mathbf{R}^d} L^I(\bfphi_t\mid\bfy) \pi(\bfphi_t)\, d\bfphi_t < \infty,\quad & \text{ for } \quad  \pi(\bfphi_t)=\pi^R(\bfphi_t), \pi^{J1}(\bfphi_t), \pi^{J2}(\bfphi_t),
\end{align*}
where $L^I(\bfphi_t\mid\bfy) \propto |\bfR_t|^{-1/2}|\bfX_t^\top \bfR_t^{-1} \bfX_t|^{-1/2} \{S^2(\bfphi_t)\}^{-(({n}_t - q_t)/2+a_t-1)}$, $\pi^R(\bfphi_t) \propto I_t^R(\bfphi_t)$, $\pi^{J1}(\bfphi_t) \propto I_t^J(\bfphi_t)$, $\pi^{J2}(\bfphi_t) \propto I_t^J(\bfphi_t) |\bfX_t^\top \bfR_t^{-1} \bfX_t|^{1/2}$. Then the condition~\eqref{eqn: condition for properiety} is satisfied by Fubini's theorem.  
\end{proof}

\subsection{Parameter Estimation}
With the above prior specification, the integrated posterior of $\bfphi$ given $\bfy$ is given by 
\begin{align} \label{eqn: posterior}
    \pi(\bfphi\mid \bfy) \propto  \prod_{t=1}^s |{\bfR}_t|^{-1/2} |{\bfX}_t^\top {\bfR}_t^{-1} {\bfX}_t|^{-1/2} \{S^2(\bfphi_t)\}^{-(({n}_t - q_t)/2 + a_t-1)} \pi(\bfphi_t),
\end{align}
where $\pi(\bfphi_t)$ refers to independent reference prior and independent Jeffreys priors. Inference based on this posterior distribution can be made in a fully Bayesian paradigm via Markov chain Monte Carlo methods. Although uncertainties in all parameters can be taken into consideration in a fully Bayesian approach, the associated computation can be too expensive in practice due to repeated evaluation of the integrated likelihood~\eqref{eqn: integrated likelihood}. 

In what follows, we focus on empirical Bayesian inference by maximizing the posterior~\eqref{eqn: posterior} to obtain the estimate of $\bfphi$. In fact, the numerical optimization can be performed for $\bfphi_t$ independently, that is, for $t=1, \ldots, s$,
\begin{align}
\begin{split}
    \hat{\bfphi}_t := \underset{\bfphi_t}{\text{argmax}} \left\{-\frac{1}{2}\ln |\bfR_t| - \frac{1}{2}\ln|\bfX_t^\top \bfR_t^{-1} \bfX_t| - \left(\frac{n_t-q_t}{2}+a_t-1\right)\ln S^2(\bfphi_t) 
     + \ln \pi(\bfphi_t)\right\},
    \end{split}
\end{align}
where the maximization step can be performed using standard optimization algorithms such as the Nelder-Mead algorithm \cite{Nocedal2006}. Once $\hat{\bfphi}_t$ is obtained, the cokriging predictions and cokriging variances can be obtained based on the posterior predictive distribution $\pi(\bfy(\bfx_0)\mid \bfy, \hat{\bfphi})$. Notice that there is no need to estimate other model parameters $\{\bfbeta, \bfgamma, \bfsigma^2\}$. If desired, these model parameters can be estimated based on the posterior distribution $\pi(\bfbeta, \bfgamma\mid \bfy, \hat{\bfphi})$ and $\pi(\bfsigma^2\mid \bfy, \hat{\bfphi})$. The detailed procedures to estimate these parameters are given in Appendix~\ref{app: parameter estimation}.

\section{Numerical Illustration} \label{sec: numerical demonstration}
The main goal of the numerical illustration is to demonstrate the predictive performance of the autoregressive cokriging model with objective priors developed in previous sections. In addition, we also include the jointly robust prior \cite{Gu2019} in the comparison, since the jointly robust prior mimics the behavior of reference priors for Gaussian process models and it is a proper prior that allows fast computation. The proposed methods have been implemented via the \textsf{R} package \texttt{ARCokrig} \cite{ARCokrig} with version 0.1.0 available at \url{https://github.com/pulongma/ARCokrig}. All the numerical examples can be reproduced via \textsf{R} code available at \url{https://github.com/pulongma/OBayesARCokrig}. The form of the jointly robust prior is 
\begin{align*}
\pi^{JR}(B_1, \ldots, B_d) = C\left( \sum_{i=1}^d C_i B_i  \right)^{a_0} \exp\left\{ - b_0 \left( \sum_{i=1}^d C_i B_i  \right)  \right\},
\end{align*}
where $B_i$'s are inverse range parameters; $C$ is a normalizing constant; $a_0>-(d+1)$, $b_0>0$ and $C_i=n^{-1/d}|x_i^{max} - x_i^{min} |$ are hyperparameters. $a_0$ is a parameter controlling the polynomial penalty to avoid singular correlation matrix and $b_0$ is a parameter controlling the exponential penalty to avoid diagonal correlation matrix. $n$ here refers to the number of model runs; $x_i^{max}, x_i^{min}$ refer to the maximum and minimum of input parameter $x_i$, respectively. \cite{Gu2019} recommends the following default settings for these parameters: $a_0=0.5-d$ and $b_0=1$. However, it was pointed out in \cite{Gu2019} that the choice of $a_0$ is an open problem and is problem-specific. In the following numerical study, we fix $b_0$ at 1, and tune the parameter $a_0$ to achieve comparable results. For the autoregressive cokriging model, independent jointly robust priors are assumed for correlation parameters across different levels of fidelity. In the following numerical comparison, the proposed cokriging predictors and cokriging variances explicitly take into account the uncertainties in estimating $\bfbeta, \bfgamma, \bfsigma^2$, while the closed-form predictive formulas in \cite{Gratiet2013} do not take into account the uncertainties in estimating $\bfbeta, \bfgamma, \bfsigma^2$, and $\bfphi$.  Numerical examples in \cite{Gratiet2013} indicate that the approach in \cite{Gratiet2013} performs better than the approaches in \cite{Kennedy2000, Qian2008} in terms of predictive accuracy and computational cost. In addition, the proposed inference approach is very similar to the plug-in MLE approach in \cite{Gratiet2013} except for the fact that objective priors and new predictive formulas are used. So, the focus here is to compare the proposed inference approach with the plug-in MLE approach in \cite{Gratiet2013}.

In the following numerical examples, the covariance function model is specified as a product form: $r(h) = \sigma^2 \prod_{i=1}^d r_i(h_i)$, where $r_i(h_i)$ can have the following forms:
\begin{align*}
r_i(h_i) &= \exp\left \{- \left(\frac{h_i}{\phi_i} \right)^{\alpha}  \right \}, \text{ Power-exponential correlation},  \\
r_i(h_i) &= \frac{2^{1-\nu}}{\Gamma(\nu)} \left( \sqrt{2\nu} \frac{h_i}{\phi_i} \right)^{\nu} \mathcal{K}_{\nu} \left( \sqrt{2\nu} \frac{h_i}{\phi_i} \right), \text{ Mat\'ern correlation}, 
\end{align*}
where $\phi_i$ is the range parameter for the $i$th input dimension. $\alpha$ is the roughness parameter in the power-exponential correlation. In practice it is often fixed at 1.9 to avoid numerical instability. $\nu>0$ is the smoothness parameter controlling the differentiability of the Gaussian processes. When $\nu$ is fixed at 5/2, the corresponding Gaussian process realizations will be twice differentiable in the mean square sense.  $\mathcal{K}_{\nu}$ is the modified Bessel function of the second kind. In all the numerical examples, the range parameter will be reparameterized with the log inverse range parameters $\xi_i:=\log(1/\phi_i)$. This parametrization facilitates robust estimation for Gaussian process emulation as shown in \cite{Gu2018, Gu2019} for reference priors and jointly robust priors. The predictive performance is measured based on root-mean-squared-prediction error (RMSE), coverage probability of the 95\% equal-tail credible interval (CVG(95\%)), and average length of the 95\% equal-tail credible interval (ALCI(95\%)). 

In the following numerical examples, the roughness parameter is always fixed at 1.9 for power-exponential correlation for the proposed approach. The plug-in MLE approach by \cite{Gratiet2013} is implemented in the \textsf{R} package \texttt{MuFiCokriging} by \cite{MuFiCokriging}. It is worth noting that \texttt{MuFiCokriging} seems to estimate the roughness parameter $\alpha$ in the power-exponential correlation and there is no way for users to fix the roughness parameter to be 1.9 in the power-exponential correlation through the package \texttt{MuFiCokriging}. It is worth noting that when the process is assumed to be rough (such as the one under the power-exponential correlation with roughness parameter 1.9), the corresponding prediction uncertainty is larger than the one based on a smoother process (such as the one under the Mat\'ern-5/2 correlation). This is because the rougher the process realization is assumed to be, the larger uncertainty we obtain in the prediction.

\subsection{Testbed with the Borehole Function} \label{sec: borehole}

The performance is investigated with the 8-dimensional borehole function that models water flow through a borehole drilled from two ground surfaces through two aquifers \cite{Harper1983}. Its fast evaluation makes it widely used for testing purposes in computer experiments \cite{Morris1993, Xiong2013}. Let $\bfx = (r_w, r, T_u, H_u, T_{\ell}, H_{\ell}, L, K_w)^\top$ be a vector of input variables in the borehole function with their physical meanings given in Appendix~\ref{app: testing function}. The response of the model is given by 
\begin{align*}
    y_h = \frac{2\pi T_u (H_u - H_{\ell})}{\log(r/r_{w}) \left[ 1 + \frac{2LT_u}{\log(r/r_w)r_w^2 K_w} + T_u/T_{\ell} \right] },
\end{align*}
and its low-fidelity output is given by 
\begin{align*}
    y_l = \frac{5 T_u (H_u - H_{\ell})}{\log(r/r_{w}) \left[ 1.5 + \frac{2LT_u}{\log(r/r_w)r_w^2 K_w} + T_u/T_{\ell} \right] }.
\end{align*}

To setup the experiment, 100 inputs are selected via Latin hypercube design with the \texttt{DiceDesign} package \cite{DiceDesign}. Then 20 inputs are randomly held out to evaluate predictive performance. The remaining 80 inputs are used to run the low-fidelity code $y_l(\cdot)$, and 30 inputs are randomly selected from these 80 inputs to run the high-fidelity code $y_h(\cdot)$. The predictive performance of the autoregressive cokriging model based on the proposed new formulas is compared with the plug-in MLE approach in \cite{Gratiet2013}. For all methods, the mean function is chosen to be constant. The covariance function is chosen to be the power-exponential correlation with roughness parameter fixed at 1.9 for the proposed approach. The results based on the Mat\'ern covariance with smoothness parameter fixed at 2.5 are reported in Appendix~\ref{app: Result with Matern}.

Figure~\ref{fig: borehole example} compares predictive means against the high-fidelity code output at 20 held-out inputs. It indicates that the proposed approach gives better prediction than the plug-in MLE approach in \cite{Gratiet2013}, since the predictive values based on the proposed method are more concentrated along the 45 degree line than the plug-in MLE approach in \cite{Gratiet2013}. Table~\ref{table: borehole example with powexp} shows that the new cokriging predictors and cokriging variances give better predictive performance than the plug-in MLE approach in \cite{Gratiet2013} in terms of RMSE and ALCI. The result based on the jointly robust prior is obtained by fixing the hyperparameter $a_0$ at 0.2 after trying several different values. It is worth noting that the proposed cokriging variances take into account uncertainties in estimating all model parameters except the range parameters, but they still give much shorter predictive intervals. This also reveals that maximizing the posterior with the uniform improper priors for correlation parameters or maximizing the concentrated restricted likelihood are less preferred than maximizing the posterior with the proposed objective priors or the jointly robust prior. 

The plug-in MLE approach in \cite{Gratiet2013} is found to yield larger predictive intervals than the proposed approach  even though uncertainties due to estimation of $\bfbeta, \bfgamma, \bfsigma^2$ are not accounted for. This is occurring because the concentrated restricted likelihood in \cite{Gratiet2013} can have nonrobust parameter estimates as discussed in Appendix~\ref{app: nonrobust estimation}, and we found in this example that several correlation parameters are estimated to be very large, resulting in a nearly singular correlation matrix.  The independent reference prior gives slightly better performance than the independent Jeffreys prior. In contrast, the jointly robust prior seems to give worse predictive performance than these two objective priors in terms of RMSE and ALCI.

\begin{figure}[htbp]
\renewcommand{\figurename}{Fig.}
\captionsetup{labelsep=period}
\begin{center}
\makebox[\textwidth][c]{ \includegraphics[width=1.0\textwidth, height=0.30\textheight]{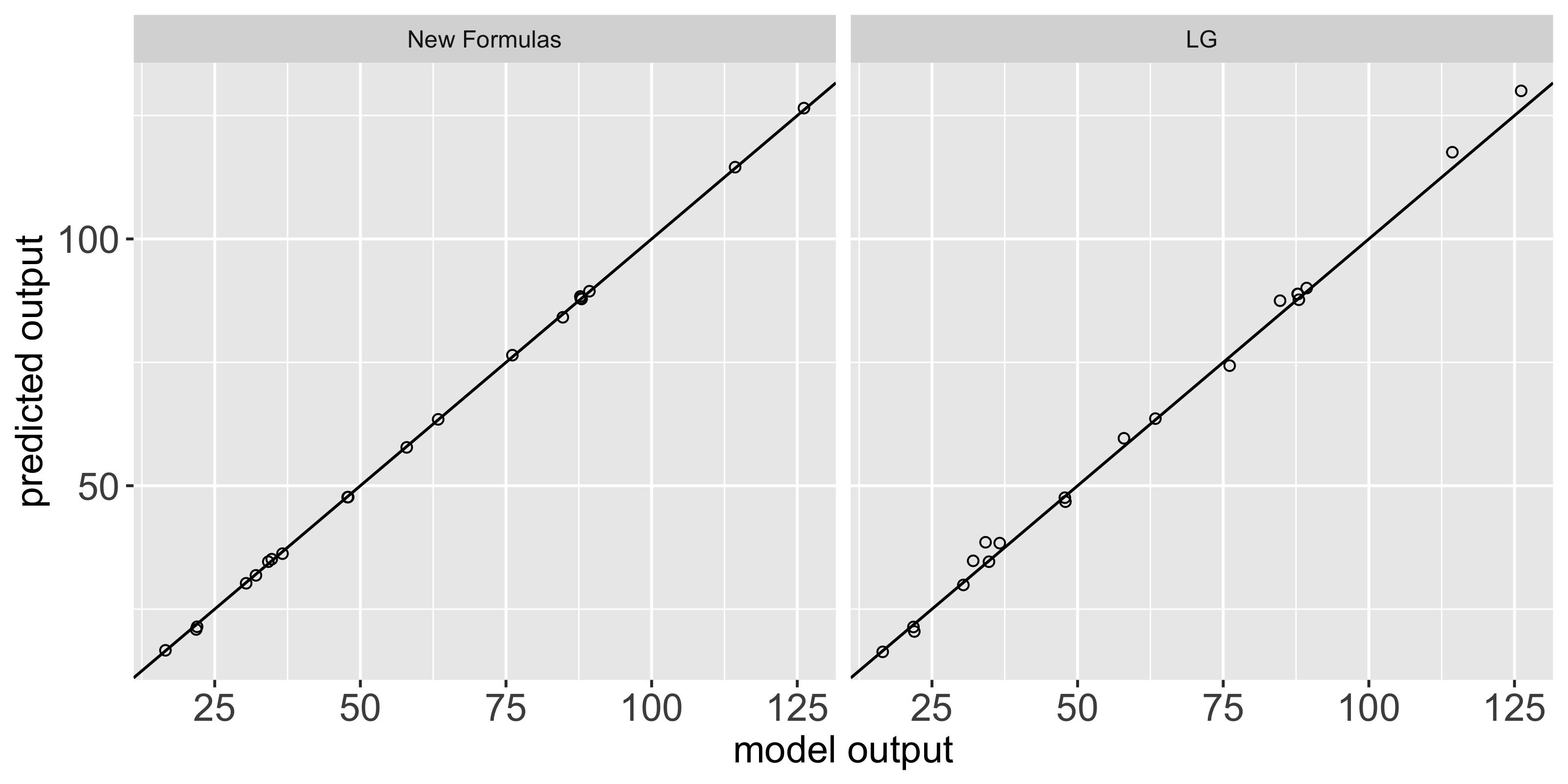}}
\caption{Prediction versus held-out data at 20 new inputs. The horizontal axis represents the held-out model output at 20 inputs, and the vertical axis represent the prediction results using the proposed new formulas (left panel) with the independent reference prior and the prediction results using the plug-in MLE approach in \cite{Gratiet2013} (right panel).}
\label{fig: borehole example}
\end{center}
\end{figure}

\begin{table}[htbp]
\centering
\normalsize
   \caption{Predictive performance at 20 held-out inputs in the autoregressive cokriging model using the proposed objective priors and using the plug-in MLE approach in Gratiet \cite{Gratiet2013} when the power-exponential correlation is used.}
     {\resizebox{1.0\textwidth}{!}{%
  \setlength{\tabcolsep}{2.0em}
   \begin{tabular}{l c c c } 
   \toprule 
   \noalign{\vskip 1.5pt}
& RMSE & CVG(95\%)   & ALCI(95\%)   \\  \noalign{\vskip 1.5pt}
\noalign{\vskip 1.5pt} \hline \noalign{\vskip 3pt}   \noalign{\vskip 1.5pt}
				 {Independent reference prior} &0.799      & 1.00 & 6.028  \\ 
				 \noalign{\vskip 4pt}				    
				 {Independent Jeffreys prior} &0.838   & 1.00   & 6.179  \\
				 \noalign{\vskip 4pt}  
		       {Jointly robust prior} & 1.000   & 1.00    &  11.61  \\
		       \noalign{\vskip 4pt} 
		       {plug-in MLE} &  2.587   & 1.00    & 19.68  \\
\noalign{\vskip 1.5pt} \bottomrule
   \end{tabular}%
   }}
   \label{table: borehole example with powexp}
\end{table}

\subsection{Application to Fluidized-Bed Processes} \label{sec: FBed}
This section studies the predictive performance of the proposed cokriging formulas with objective priors for the fluidized-bed process experiment analyzed in \cite{Qian2008, Gratiet2013}. The computer model named ``Topsim" simulates the temperature of the steady-state thermodynamic operation point for a fluidized-bed process based on eight physical parameters: fluid velocity of the fluidization air, temperature of the air from the pump, flow rate of the coating solution, temperature of the coating solution, coating solution dry matter content, pressure of atomized air, room temperature, and humidity. \cite{Dewettinck1999} consider 28 different process conditions with coating solution used for distilled water (i.e., coating solution dry matter content is 0) and the room temperature at 20$^\circ$C. For each input configuration, one physical experiment $T_{exp}$ and three computer model runs ($T_1, T_2, T_3)$ were conducted, where $T_{exp}$ is the experimental response, $T_3$ is the most accurate code modeling the experiment, $T_2$ is a simplified version of $T_3$, and $T_1$ is the lowest accurate code modeling the experiment. The six inputs and corresponding outputs $T_1, T_2, T_3$, and $T_{exp}$ for these 28 runs are given in \cite{Qian2008}. The following numerical study mainly follows the procedure in \cite{Gratiet2013}, and compares the proposed approach with the plug-in MLE approach in \cite{Gratiet2013}, since Gratiet \cite{Gratiet2013} demonstrates that the plug-in MLE approach performs better than the approach in \cite{Qian2008}. Following \cite{Gratiet2013}, each input parameter is also scaled to the unit interval $[0, 1]$.

The predictive performance of a 2-level cokriging model is investigated based on 20 randomly selected $T_{exp}$ runs and all 28 $T_2$ runs. The remaining eight $T_{exp}$ runs are used for model validation. The mean function is chosen to be constant and the covariance function is chosen to be the power-exponential correlation with roughness parameter fixed at 1.9. The results based on the Mat\'ern with smoothness parameter 2.5 at each level are reported in Appendix~\ref{app: Result with Matern}. In the jointly robust prior, the hyperparameter $a_0$ chosen to be the default setting seems to give better performance than the other settings. The results in Table~\ref{table: two-level example with powexp} show that the 2-level cokriging model under the objective priors and jointly robust priors yields much smaller RMSE and ALCI than the plug-in MLE approach in \cite{Gratiet2013}, while the 2-level cokriging models with all different priors have the empirical coverage probability smaller than the nominal level 0.95.  We also found that the roughness parameters in the power-exponential correlation are 2 in the plug-in MLE approach via the package \texttt{MuFiCokriging}. This indicates that the corresponding process realizations will be infinitely differentiable. As expected, the corresponding prediction uncertainty is smaller in Table~\ref{table: two-level example with powexp} than the one under the Mat\'ern-5/2 correlation in Table~\ref{table: two-level example}. 

Then we investigate the predictive performance of a 3-level cokriging model. To setup the design, we used 10 $T_{exp}$ runs (with the row number given by 1, 3, 8, 10, 12, 14, 18, 19, 20, 27 in Table 4 in \cite{Qian2008}), 20 $T_3$ runs (with the row number given by 1, 2, 3, 5, 6, 7, 8, 9, 10, 11, 12, 13, 14, 16, 18, 19, 20, 22, 24, 27), and all the 28 $T_1$ runs. The remaining 18 $T_{exp}$ runs are used for model validation. Table~\ref{table: three-level example with powexp} shows that the 3-level cokriging model gives much smaller RMSE and ALCI with independent reference prior, independent Jeffreys prior, and jointly robust prior than that with the plug-in MLE approach in \cite{Gratiet2013}. For the jointly robust prior, the hyperparameter $a_0$ chosen to be the default setting gives better results under other settings. Meanwhile, the objective priors and the jointly robust prior result in much shorter predictive intervals than the plug-in MLE approach in \cite{Gratiet2013}. Notice that the jointly robust prior leads to better predictive performance than the independent reference prior in terms of RMSE and ALCI. The difference of predictive performance between the independent reference prior and the independent Jeffreys prior is subtle. 

To briefly summarize, the objective priors and the jointly robust prior yield much better predictive performance than the plug-in MLE approach in \cite{Gratiet2013} in the 2-level cokriging model and the 3-level cokriging model. It is again worth noting that the approach in \cite{Gratiet2013} always gives larger predictive uncertainties than the proposed approach in all these examples. This is due to the fact that the estimated correlation parameters in \cite{Gratiet2013} are not robust, since several correlation parameters are estimated near zero. Such nonrobust estimates can mess up with prediction results. This suggests that inference with the objective priors should always be preferred over inference with the uniform improper priors or the plug-in MLE approach \cite{Gratiet2013} when reliable experts' opinions do not exist. In addition, the proposed cokriging predictors and cokriging variances are recommended for predictive inference in autoregressive cokriging models.


\begin{table}[htbp]
\centering
\normalsize
   \caption{Predictive performance under 2-level cokriging based on $T_{2}$ and $T_{exp}$ using the proposed objective priors and the plug-in MLE approach in \cite{Gratiet2013} when the power-exponential correlation is used.}
     {\resizebox{1.0\textwidth}{!}{%
  \setlength{\tabcolsep}{2.0em}
   \begin{tabular}{l c c c } 
   \toprule 
   \noalign{\vskip 1.5pt}
& RMSE & CVG(95\%)   & ALCI(95\%)   \\  \noalign{\vskip 1.5pt}
\noalign{\vskip 1.5pt} \hline \noalign{\vskip 3pt}   \noalign{\vskip 1.5pt}
{Independent reference prior} &0.551      & 0.88 & 2.239  \\ 
				 \noalign{\vskip 4pt}				    
{Independent Jeffreys prior} &0.542   & 0.88   & 2.236  \\
				 \noalign{\vskip 4pt}  
{Jointly robust prior} & 0.530   & 0.88    &  2.309  \\
		       \noalign{\vskip 4pt} 
 {plug-in MLE} &  1.968  & 0.88    & 6.028  \\
\noalign{\vskip 1.5pt} \bottomrule
   \end{tabular}%
   }}
   \label{table: two-level example with powexp}
\end{table}

\begin{table}[htbp]
\centering
\normalsize
   \caption{Predictive performance under 3-level cokriging based on $T_{1}$, $T_3$, and $T_{exp}$ using the proposed objective priors and the plug-in MLE approach in \cite{Gratiet2013} when the power-exponential correlation is used.}
     {\resizebox{1.0\textwidth}{!}{%
  \setlength{\tabcolsep}{2.0em}
   \begin{tabular}{l c c c } 
   \toprule 
   \noalign{\vskip 1.5pt}
& RMSE & CVG(95\%)   & ALCI(95\%)   \\  \noalign{\vskip 1.5pt}
\noalign{\vskip 1.5pt} \hline \noalign{\vskip 3pt}   \noalign{\vskip 1.5pt}
{Independent reference prior} &0.945      & 1.00 & 7.857  \\ 
				 \noalign{\vskip 4pt}				    
{Independent Jeffreys prior} &1.086   & 1.00   & 7.645  \\
				 \noalign{\vskip 4pt}  
{Jointly robust prior} & 0.874  & 1.00    &  5.620  \\
		       \noalign{\vskip 4pt} 
{plug-in MLE} &  8.844  & 0.72    & 19.04  \\
\noalign{\vskip 1.5pt} \bottomrule
   \end{tabular}%
   }}
   \label{table: three-level example with powexp}
\end{table}

\section{Discussion} \label{sec: discussion}
This article presents a unifying view in making prediction and parameter estimation in a computationally efficient way in the sense that the computational cost of both prediction and parameter estimation in an $s$-level cokriging model is the same as that in $s$ independent kriging models. The formulas of the predictive distributions account for uncertainties in all model parameters except correlation parameters. The objective Bayesian analysis performed in the autoregressive cokriging model can be used as a default choice when prior information is challenging to obtain. In addition, the independent reference prior can also encourage robust estimation of correlation parameters. 

The independent reference prior and independent Jeffreys priors are shown to yield proper posterior distributions. The predictive performance under these objective priors and independent jointly robust priors are also compared based on frequentist properties under various numerical studies. The numerical examples show that the objective priors and the jointly robust prior yield very similar predictive performance in 2-level and 3-level cokriging models. We also found that the jointly robust prior can provides better predictive performance than the independent reference prior sometimes, but this requires tuning its hyperparameters. The determination of optimal values for its hyperparameters is still an open question. 

It is worth noting that the reference prior should be preferred over the Jeffreys prior especially when the dimension of input space is large as pointed out by several previous works \cite{Berger2001, Berger2015}. As we assume the priors are independent across different levels of fidelity, the behavior of a joint prior across all levels is solely characterized by the behavior of priors at each level of fidelity. At each level of fidelity, the reference prior mass moves from the smaller values of range parameters to the larger values of range parameters for each input dimension. This implies that a degenerate case (near-diagonal correlation matrix) should be avoided. When input dimension increases, the chance that at least one range parameter is estimated to be small increases, if the prior mass does not change along with the input dimension, and consequently, the chance that near-diagonal correlation matrix also increases. The reference prior adapts to the increase of the dimension by concentrating more prior mass at larger range parameters when input dimension increases. Similar result is also pointed out by \cite{Gu2018} for Gaussian processes.

When designs are not hierarchically nested, there is no closed-form expression for the marginal likelihood function, and hence objective Bayesian analysis could be very challenging. However, the independent jointly robust prior could be a promising choice for this situation, since it has comparable predictive performance and it is a proper prior that allows fast computation. This has been used in \cite{Ma2019PPCokriging} for parameter estimation in a cokriging model that emulates high-dimensional output from multiple computer models. Recently, \cite{Ma2019Cov} propose a new class of correlation functions, which has a polynomial decaying tail and allows any degrees of smoothness of process realization. This correlation function seems to perform better than the widely-used Mat\'ern class. Thus, it would be also interesting to develop objective priors for correlation parameters in this correlation function.


\appendix


\section{Computational Details}
\subsection{Gaussian Process Regression} \label{app: GPR}
Here some useful facts are reviewed in Gaussian process regression and they will be repeatedly used in the derivation of formulas in this paper. The notations in this section shall not be confused with notations used in other places of this paper. Suppose that $z(\cdot) \sim \mathcal{GP}(\bfh^\top(\cdot)\bfbeta,\, \sigma^2r(\cdot, \cdot|\bfphi))$ with output values $\bfz:=(z(\bfx_1), \ldots, z(\bfx_n))^\top$ taken at input values $\mathcal{X}:=\{\bfx_i \}_{i=1}^n$. For any new input $\bfx_0$, we have  
\begin{align*}
z(\bfx_0) \mid \bfz, \bfbeta, \sigma^2, \bfphi \sim \mathcal{N}(\bfh^\top(\bfx_0)\bfbeta,\, \sigma^2\{r(\bfx_0, \bfx_0|\bfphi) - 
\bfr^\top(\bfx_0) \bfR^{-1} \bfr(\bfx_0) \} ),
\end{align*}
where $\bfH:=[\bfh(\bfx_1), \ldots, \bfh(\bfx_n)]^\top$. $\bfR:=[r(\bfx_i, \bfx_j)]_{i,j=1}^n$. $\bfr(\bfx_0):=r({\calX}, \bfx_0\mid \bfphi)$.

Integrating out parameters $\bfbeta, \sigma^2$ with respect to the prior $\pi(\bfbeta, \sigma^2) \propto (\sigma^2)^{-1}$ yields the predictive distribution of $z(\bfx_0)$ given $\bfz$ and $\bfphi$:
\begin{align*}
z(\bfx_0) \mid \bfz, \bfphi \sim t_{n-q}(\hat{z}(\bfx_0),\, \hat{\sigma}^2c^*)
\end{align*}
with  
\begin{align*}
\begin{split}
\hat{z}(\bfx_0) &= \bfh^\top(\bfx_0) \hat{\bfbeta} + \bfr^\top(\bfx_0) \bfR^{-1} (\bfz - \bfH \hat{\bfbeta})\\
\hat{\sigma}^2 & :=(\bfz-\bfH\hat{\bfbeta})^{\top}\bfR^{-1}(\bfz-\bfH\hat{\bfbeta}) / (n-q),\\
c^* & :=r(\bfx_0,\bfx_0|\bfphi)-\bfr^\top(\bfx_0)\bfR^{-1}\bfr(\bfx_0) \\
&+ [\bfh(\bfx_0)-\bfH^{\top}\bfR^{-1}\bfr(\bfx_0)]^{\top}(\bfH^{\top}\bfR^{-1}\bfH)^{-1}[\bfh(\bfx_0)-\bfH^{\top}\bfR^{-1}\bfr(\bfx_0)],
\end{split}
\end{align*}
where $\hat{\bfbeta}: = (\bfH^\top \bfR^{-1} \bfH)^{-1}\bfH^\top \bfR^{-1} \bfz$. $q$ is the number of columns in $\bfH$.

Given the prior $\pi(\bfbeta, \sigma^2, \bfphi) \propto \pi(\bfphi) (\sigma^2)^{-a}$, the integrated likelihood in the Gaussian process model is   
\begin{align*}
\int (\sigma^2)^{-a} \mathcal{N}(\bfH\bfbeta, \sigma^2 \bfR) d (\bfbeta, \sigma^2) \propto 
|\bfR|^{-1/2} |\bfH^\top \bfR^{-1} \bfH|^{-1/2} (S^2)^{-\{(n-q)/2+a-1\}};
\end{align*}
here $S^2:=\bfz^\top \bfQ \bfz$ with $\bfQ:=\bfR^{-1}\bfP$ and $\bfP:=\mathbf{I} - \bfH (\bfH^\top \bfR^{-1} \bfH)^{-1} \bfH^\top \bfR^{-1}$. This formula will be used to derive the integrated likelihood in autoregressive cokriging models in Appendix~\ref{app: integrated likelihood}.

\subsection{Integrated Likelihood} \label{app: integrated likelihood}
Using the representation of the marginal likelihood function~\eqref{eqn: sampling dist} and the prior~\eqref{eqn: general prior}, it follows that 
\begin{align*}
& \int L({\bfy}\mid\bfbeta,\bfgamma,\bfsigma^{2},\bfphi)\pi(\bfbeta,\bfgamma,\bfsigma^{2},\bfphi)\,d(\bfbeta,\bfgamma,\bfsigma^{2}) \\
&=  \int \pi({\bfy}_{1}\mid\bfbeta_{1},\sigma_{1}^{2},\bfphi_{1}) \left\{ \prod_{t=2}^{s}\pi({\bfy}_{t}\mid{\bfy}_{t-1},\gamma_{t-1},\bfbeta_{t},\sigma_{t}^{2},\bfphi_{t}) \right\} \frac{\pi(\bfphi)}{\prod_{t=1}^s(\sigma^2_t)^{a_t}} d(\bfbeta,\bfgamma,\bfsigma^{2})  \\
&= \int (\sigma^2_1)^{-a_1} \mathcal{N}(\bfH_1\bfbeta_1, \sigma^2_1 \bfR_1) \left\{\prod_{t=2}^s \mathcal{N}({\bfH}_{t}\bfbeta_{t} + {W}_{t-1}\gamma_{t-1},\,\sigma_{t}^{2} {\bfR}_t) (\sigma^2_t)^{-a_t}\right\} d(\bfbeta,\bfgamma,\bfsigma^{2}) \times \pi(\bfphi) \\
&= \int (\sigma^2_1)^{-a_1} \mathcal{N}(\bfH_1\bfbeta_1, \sigma^2_1 \bfR_1) d(\bfbeta_1, \sigma^2_1)   \\
&\quad \times \left\{ \prod_{t=2}^s \int \mathcal{N}({\bfH}_{t}\bfbeta_{t} + {W}_{t-1}\gamma_{t-1},\,\sigma_{t}^{2} {\bfR}_t) (\sigma^2_t)^{-a_t}  d(\bfbeta_t, \gamma_{t-1}, \sigma^2_t) \right\} \times \pi(\bfphi) \\
 & = L^I(\bfphi\mid \bfy) \pi(\bfphi),
\end{align*}
where the last equation follows by applying the facts in Appendix~\ref{app: GPR} for $t=2, \ldots, s$.

\section{Proofs} \label{app: Fisher Info}

\subsection{Proof of Lemma~\ref{lem: conditional predictive dists}} \label{app: conditional predictive dist}
As $\{\calX_t\cup\{\bfx_0\}: t=1, \ldots, s\}$ forms a collection of nested design, it follows from the cokriging model~\eqref{eqn: co-kriging model} that 
\begin{align*}
    \pi(\bfy(\bfx_0), \bfy | \bfbeta, \bfgamma, \bfphi) &= \pi(y_1(\bfx_0), \bfy_1 | \bfbeta_1, \sigma^2_1, \bfphi_1) \prod_{t=2}^s \pi(y_t(\bfx_0), \bfy_t | y_{t-1}(\bfx_0), \bfy_{t-1}, \bfbeta_t, \gamma_{t-1}, \sigma^2_t, \bfphi_t), 
\end{align*}
where each joint distribution on the right hand side is a multivariate normal distribution. In particular, the joint distribution of $y_t(\bfx_0), \bfy_t$ given $y_{t-1}(\bfx_0), \bfy_{t-1}, \bfbeta_t, \gamma_{t-1}, \sigma^2_t, \bfphi_t$ is 
\begin{align*}
& \begin{pmatrix}
y_t(\bfx_0) \\
 \bfy_t
\end{pmatrix}
\bigg|\, y_{t-1}(\bfx_0), \bfy_{t-1}, \bfbeta_t, \gamma_{t-1}, \sigma^2_t, \bfphi_t \\
& \sim \mathcal{N}\left( 
\begin{pmatrix}  \bfh^\top(\bfx_0) \bfbeta_t + y_{t-1}(\bfx_0) \gamma_{t-1} \\  
\bfX_t \bfbeta_t 
\end{pmatrix}, 
\begin{pmatrix}
\sigma^2_t r(\bfx_0, \bfx_0| \bfphi_t) & \sigma^2_t \bfr_t^\top(\bfx_0) \\
\sigma^2_t \bfr_t(\bfx_0) & \sigma^2_t \bfR_t
\end{pmatrix}
\right).
\end{align*}
Using the facts in Appendix~\ref{app: GPR} yields the conditional distribution of $y_t(\bfx_0)$ given $y_{t-1}(\bfx_0),$ $\bfy_t, \bfy_{t-1}, \bfbeta_t, \gamma_{t-1}, \sigma^2_t, \bfphi_t$ as a normal distribution. Combining these conditional distributions for $t=1, 2, \ldots, s$ yields that 
\begin{align*}
    \pi(\bfy(\bfx_0) |\bfy, \bfbeta, \bfgamma, \bfphi) &= \pi(y_1(\bfx_0) | \bfy_1, \bfbeta_1, \sigma^2_1, \bfphi_1) \prod_{t=2}^s \pi(y_t(\bfx_0) | \bfy_t, \bfy_{t-1}, y_{t-1}(\bfx_0), \bfbeta_t, \gamma_{t-1}, \sigma^2_t, \bfphi_t), 
\end{align*}
where each predictive distribution on the right hand side is also a normal distribution. For each of these normal distributions, integrating out the location-scale parameters $\bfbeta_t, \gamma_{t-1}, \sigma^2_t$ with the prior \eqref{eqn: general prior} yields the formulas given in Lemma~\ref{lem: conditional predictive dists} according to Appendix~\ref{app: GPR}. 

\subsection{Proof of Theorem~\ref{thm: predictive mean and variance}} \label{app: predictive mean and variance}
Notice that for $t=1, \bff_1(\bfx_0):=h_1(\bfx_0)$ and $\bfX_1:=\bfH_1$; for $t=2, \ldots, s$, $\bff_t^\top(\bfx_0) = [h_t^\top(\bfx_0),$ $\hat{y}_{t-1}(\bfx_0)]$ and $\bfX_t:=[\bfH_t, y_{t-1}(\calX_t)]$. 
The formula for the cokriging predictor at level $t$ follows from the law of total expectation as follows:
\begin{align*}
     E_{y_t(\bfx_0)|\bfy, \bfphi}\{y_t(\bfx_0) \} &=  E_{[y_{t-1}(\bfx_0)|\bfy, \bfphi]}\{E_{[y_t(\bfx_0)|\bfy, \bfphi, y_{t-1}(\bfx_0)]}[y_t(\bfx_0)] \} \\
    &=E_{[y_{t-1}(\bfx_0)|\bfy, \bfphi]}\{\bff_t^\top(\bfx_0)\hat{\bfb}_t+\bfr_t^\top(\bfx_0)\bfR_t^{-1}(\bfy_{t}-\bfX_t\hat{\bfb}_t) \} \\
    &= E_{[y_{t-1}(\bfx_0)|\bfy, \bfphi]}\{\bff_t^\top(\bfx_0)\}\hat{\bfb}_t+\bfr_t^\top(\bfx_0)\bfR_t^{-1}(\bfy_{t}-\bfX_t\hat{\bfb}_t)  
    = \hat{y}_t(\bfx_0).
\end{align*}
As the expectation is taken with respect to the distribution $\pi(y_{t-1}(\bfx_0)| \bfy, \bfphi)$, we have $\hat{y}_{t-1}(\bfx_0)$ $= E_{[y_{t-1}(\bfx_0)|\bfy, \bfphi]}\{y_{t-1}(\bfx_0) \}$. As only the term $\bff_t(\bfx_0)$ contains $y_{t-1}(\bfx_0)$, taking conditional expectation yields the last equation.

The formula for the cokriging variance at level $t$ follows from the law of total variance: 
\begin{align*}
   Var\{y_t(\bfx_0)|\bfy, \bfphi \} &= Var\{ E[y_t(\bfx_0)|\bfy, \bfphi, y_{t-1}(\bfx_0)] \mid \bfy, \bfphi \} \\
   &\quad + E\{ Var[y_t(\bfx_0)|\bfy, \bfphi, y_{t-1}(\bfx_0)] \mid \bfy, \bfphi \},
\end{align*}
with 
\begin{align*}
    Var\{ E[y_t(\bfx_0)|\bfy, \bfphi, y_{t-1}(\bfx_0)] \mid \bfy, \bfphi \} &= Var\{\hat{\gamma}_{t-1} y_{t-1}(\bfx_0) | \bfy, \bfphi\} = \hat{\gamma}_{t-1}^2 \hat{v}_{t-1}(\bfx_0), \\
    E\{ Var[y_t(\bfx_0)|\bfy, \bfphi, y_{t-1}(\bfx_0)] \mid \bfy, \bfphi \} &= \frac{n_t-q_t}{n_t-q_t-2} E\{\hat{\sigma}_t^2 c_t^* | \bfy, \bfphi \},
\end{align*}
where the last equation follows from the property of Student $t$ distribution.  Applying the facts in Appendix~\ref{app: GPR} yields that $E(c_t^*|\bfy, \bfphi) = r(\bfx_0, \bfx_0|\bfphi_t) - \bfr_t^\top(\bfx_0) \bfR_t^{-1} \bfr_t(\bfx_0) + \kappa_t$.

\subsection{Proof of Lemma~\ref{lem: behavior of integrated likelihood}} \label{app: nu expression}
Let  the integrated likelihood at level $t$ be 
\begin{align*}
\begin{split}
    L^I(\bfphi_t\mid \bfy) 
    & \propto  |{\bfR}_t|^{-1/2} |{\bfX}_t^\top {\bfR}_t^{-1} {\bfX}_t|^{-1/2} \{S^2(\bfphi_t)\}^{-(({n}_t - q_t)/2+a_t-1)}. 
\end{split}
\end{align*}
\begin{itemize}[noitemsep,topsep=0pt]
    \item[(i)] The continuity of $L^I(\bfphi_t\mid \bfy)$ follows from the continuity of $r(u/\phi_{t,\ell})$ as a function of $\phi_{t, \ell}$.  Notice that $\bfR_t=\bfR_{t,1}\circ \bfR_{t,2} \circ \ldots \circ \bfR_{t,d}$, where $\bfR_{t,\ell}$ is the $n_t\times n_t$ correlation matrix for the $\ell$th input dimension with correlation parameter $\phi_{t,\ell}$ for $\ell=1, \ldots, d$. ``$\circ$'' denotes the elementwise product. As $\phi_{t,\ell} \to 0^+$, $\bfR_{t,\ell} \to \mathbf{I}$ and hence $\bfR_{t} \to 
    \bfR_{t,-\ell}:= \bfR_{t,1}\circ \ldots \circ \bfR_{t,\ell-1} \circ \mathbf{I} \circ \bfR_{t,\ell+1} \circ \ldots \circ \bfR_{t,d}$.
     Thus, as $\phi_{t,\ell}\to 0^+$, $L^I(\bfphi_t\mid \bfy) \to c_0^t$ up to an multiplicative constant with 
     \begin{align*}
      c_0^t := |{\bfR}_{t, -\ell}|^{-1/2} |{\bfX}_t^\top {\bfR}_{t,-\ell}^{-1} {\bfX}_t|^{-1/2} \{S^2(\bfphi_{t, -\ell})\}^{-(({n}_t - q_t)/2+a_t-1)},
     \end{align*}
     where $S^2(\bfphi_{t, -\ell})$ is defined as $S^2(\bfphi_t)$ with $\bfR_t$ replaced by $\bfR_{t,-\ell}$. As $c_0^t>0$, the integrated likelihood $L^I(\bfphi_t \mid \bfy)$ at level $t$ exists and is greater than zero when $\phi_{t,\ell}\to 0^+$.
     
     Notice that for any fixed $\bfphi_{t}$, we have that $L^I(\bfphi_t \mid \bfy)$ as a fixed positive constant.  So, $L^I(\bfphi \mid \bfy) = \prod_{t=1}^s L^I(\bfphi_t\mid \bfy) $ goes to a fixed and positive constant as $\phi_{t, \ell} \to 0^+$, since the quantity $L^I(\bfphi_t\mid \bfy) \to c_0^t$ up to a multiplicative constant and the quantities $L^I(\bfphi_k\mid \bfy)$ with $k=1, \ldots, t-1, t+1,\ldots s$ are fixed and positive constants. Therefore, if $\exists \ell$ such that $\phi_{t,\ell} \to 0^+$ for at least one $t$, $\lim_{\phi_{t,\ell} \to 0^+} L^I(\bfphi \mid \bfy) >0$.
     
    \item[(ii)] It follows from \cite{Gu2018} that if $\phi_{t,\ell} \to \infty$ for all $\ell$ and $t$, the integrated likelihood at level $t$ satisfies 
    \begin{align*}
        L^I(\bfphi_{t} \mid \bfy) =  
        \begin{cases}
            O\left(  \left(\sum_{\ell=1}^d \nu_{t,\ell}(\phi_{t,\ell})\right)^{a_t-1/2} \right), & \mathbf{1} \notin  \mathcal{C}(\bfX_t), \\
            O\left(  \left(\sum_{\ell=1}^d \nu_{t,\ell}(\phi_{t,\ell} )\right)^{a_t-1} \right), & \mathbf{1} \in \mathcal{C}(\bfX_t).
        \end{cases}
    \end{align*} 
    As $L^I(\bfphi \mid \bfy) = \prod_{t=1}^s L^I(\bfphi_t\mid \bfy)$, the results in Lemma~\ref{lem: behavior of integrated likelihood} follow immediately.  
\end{itemize}

\subsection{Proof of Theorem~\ref{thm: reference prior}} \label{app: reference prior}
Arranging the parameters in the order $\bfvartheta:=(\bfvartheta_1, \ldots, \bfvartheta_s)$ with $\bfvartheta_t:=(\sigma^2_t, \bfphi_t^\top)^\top$, the Fisher information matrix $I^I(\bfvartheta_1, \ldots, \bfvartheta_s)$ is computed from $\ell^I(\bfvartheta \mid \bfy)$, whose $(i,j)$ entry is 
\begin{align}
    [I^I(\bfvartheta \mid \bfy)]_{t,k} = E\left\{ \frac{\partial}{\partial \vartheta_{i}} \ell^I(\bfvartheta \mid \bfy) \times \frac{\partial}{\partial \vartheta_{j}} \ell^I(\bfvartheta \mid \bfy)  \right\}.
\end{align}
Differentiation with respect to $\sigma^2_t, \phi_{t,\ell}$ yields that 
\begin{align*}
    \begin{split}
        \frac{\partial}{\partial \sigma^2_t} \ell^I(\bfvartheta \mid \bfy) = \frac{S^2_t-E(S^2_t)}{2\sigma^4},\quad & 
        \frac{\partial}{\partial\phi_{t,\ell}} \ell^I(\bfvartheta \mid \bfy) = 
        \frac{\Sigma^{\ell}_t - E(\Sigma^{\ell}_t)}{2\sigma^2_t},
    \end{split}
\end{align*}
where $S^2_t:=\bfy_t^\top \bfQ_t\bfy_t$ with $S^2_t/\sigma^2_t \sim \chi_{n_t-q_t}^2$. $\Sigma^{\ell}_t$ is quadratic form on $\bfP_t\bfy_t\sim N(\mathbf{0}, \sigma^2_t\bfP_t\bfR_t)$ associated with the matrix $\bfR_t^{-1} \dot{\bfR}_t^{\ell} \bfR^{-1}$, where $\dot{\bfR}_t^{\ell}=\frac{\partial}{\partial \phi_{t,\ell}} \bfR_t$ is element-wise differentiation. 
Using results in \cite{Berger2001}, the $(t,t)$ block diagonal matrix in the Fisher information matrix $I^I(\bfvartheta)$ is $I^R_t(\bfphi_t)$.

\section{Nonrobust Estimation} \label{app: nonrobust estimation} 
 \subsection{Posterior in Qian \cite{Qian2008}} \label{app: Qian}
 This section gives an example to show that posterior impropriety is occurring when vague priors in \cite{Qian2008} are chosen. As an illustrating example, we only discuss the posterior for correlation parameters at the first level. According to \cite{Qian2008}, the following priors are assumed:
\begin{align*}
\pi(\bfbeta \mid \sigma^2_1) &\sim \mathcal{N}(\bfu_1, v_1\mathbf{I} \sigma^2_1), \\
\pi(\sigma^2_1) &\sim \mathcal{IG}(\alpha_1, \gamma_1^0), \\
\pi(\phi_{1,\ell}) &\sim \text{Gamma}(a_1^0, b_1^0), \ell=1, \ldots, d,
\end{align*}
where $\bfu_1, v_1, \alpha_1, \gamma_1^0, a_1^0, b_1^0$ are hyperparameters. 
With similar notations in \cite{Qian2008}, the posterior distribution of $\bfphi_1$ is,  
 $$ \pi(\bfphi_1 \mid \bfy_1) \propto \pi(\bfphi_1) |\bfR_1|^{-1/2} |\bfA_1|^{-1/2} |^{-1/2} \left\{ \gamma_1^0 + \frac{4c_1 - \bfB^\top_1\bfA_1^{-1} \bfB_1}{8}  \right \}^{-(\alpha_1+n/2)} ,$$
 where $\bfA_1 = v_1^{-1} \mathbf{I} + \bfH_1^\top \bfR_1^{-1} \bfH_1$, $\bfB_1=-2v_1^{-1} \bfu_1 - 2\bfH_1^\top \bfR_1^{-1} \bfy_1$, $c_1= v_1^{-1} (\bfu^\top_1 \bfu_1) + \bfy_1^\top \bfR_1^{-1} \bfy_1$. 

Thus, when $v_1 \to \infty, \alpha_1\to 0, \gamma_1^0\to 0$ such that priors for $\bfbeta_1$ and $\sigma^2_1$ become vague, the marginal posterior of $\bfphi_1$ will be proportional to the product of the prior $\prod_{\ell=1}^d \mathcal{IG}(\bfphi_{1, \ell} \mid a_1^0, b_1^0)$ and the marginal likelihood $L(\bfphi_1\mid \bfy)=|\bfR_1|^{-1/2} |\bfH_1^\top \bfR_1^{-1} \bfH_1|^{-1/2} (S^2)^{-n/2}$, where $S^2=\bfy_1^\top\bfQ_1 \bfy_1$ and $\bfQ_1=\bfR_1^{-1} - \bfR_1^{-1}\bfH_1(\bfH_1^\top \bfR_1^{-1}$ $\bfH_1)^{-1}\bfH_1^\top \bfR_1^{-1}$. This posterior concentrates all its mass near 0 as $a_1^0 \to 0$ and $b_1^0\to 0$, resulting in nonrobust estimation according to Lemma 3.3 in \cite{Gu2018}. 
 
 \subsection{Marginal likelihood in Gratiet \cite{Gratiet2013}} \label{app: Gratiet}
 This section shows that the concentrated restricted likelihood can be maximized either at zero or infinity when noninformative priors or informative priors (when they are chosen to be vague) in \cite{Gratiet2013} are used. Gratiet \cite{Gratiet2013} considers two different types of priors for $\bfbeta, \bfgamma, \bfsigma^2$: \emph{noninformative priors} and \emph{informative priors}. The noninformative priors are chosen to be
\begin{align*}
\pi(\bfbeta_1 \mid \sigma^2_1, \bfphi_1) \propto 1, & & \pi(\sigma^2_1) \propto 1/\sigma^2_1, \\
\pi(\bfbeta_t, \bfgamma_{t-1} \mid \sigma^2_t, \bfphi_t) \propto 1, & &\pi(\sigma^2_t) \propto 1/\sigma^2_t, t=2, \ldots, s. 
\end{align*} and 
the informative priors in \cite{Gratiet2013} are chosen to be
\begin{align*}
\pi(\bfbeta_1 \mid \sigma^2_1, \bfphi_1) \sim \mathcal{N}(\bfb_1^0, \sigma^2_1 \bfV_1^0), & & \pi(\sigma^2_1 \mid \bfphi_1) \sim \mathcal{IG}(\alpha_1^0, \gamma_1^0), \\
\pi((\bfbeta_t, \gamma_{t-1}) \mid \sigma_t^2, \bfphi_t) \sim \mathcal{N}(\bfb_t^0, \sigma^2_t \bfV_t^0), & & \pi(\sigma^2_t \mid \bfphi_t) \sim \mathcal{IG}(\alpha_t^0, \gamma_t^0).
\end{align*}
Without further assuming a prior for $\bfphi$, Gratiet \cite{Gratiet2013} proposes to maximize the following concentrated restricted likelihood:
\begin{align*}
L_1(\bfphi_1 \mid \bfy, \hat{\sigma}^2_1) &\propto |\bfR_1|^{-1/2} (\hat{\sigma}^2_1)^{-(n_1-p_1)/2}, \\
L_t(\bfphi_t \mid \bfy, \hat{\sigma}^2_t) &\propto |\bfR_t|^{-1/2} (\hat{\sigma}^2_t)^{-(n_1-p_1-1)/2}, t=1, \ldots, s.
\end{align*}
For noninformative priors, $\hat{\sigma}^2_t \propto S^2(\bfphi_t)$. According to Lemma 3.3 in \cite{Gratiet2013}, these marginal likelihood functions can have modes at $\bfR = \mathbf{I}_n$ and $\bfR = \mathbf{1}_n \mathbf{1}_n^\top$, resulting in nonrobust estimates for $\bfphi$. For informative priors, the expression for $\hat{\sigma}^2_t$ is of the following form:
\begin{align*}
\hat{\sigma}^2_t \propto \gamma_t^0 + (\bfb_t - \bar{\bfb}_t)^\top \{ \bfV_t + (\bfX_t^\top \bfR^{-1} \bfX_t)^{-1} \}^{-1} (\bfb_t - \bar{\bfb}_t) + S^2(\bfphi_t),
\end{align*}
where $\bfb_t :=(\bfbeta_t, \gamma_{t-1})^\top$ with $\gamma_0:=0$. $\bar{\bfb}_t$ is the generalized least square estimate for $\bfb_t$. 
When $\gamma_t^0 \to 0$ and $\bfV_t^{-1} \to \mathbf{0}$, we have $\hat{\sigma}^2_t \propto S^2(\bfphi_t)$. This reduces to the case when noninformative priors are used. Thus, estimates of the parameters $\bfphi_t$ can be nonrobust. It is worth noting that \cite{Gratiet2013} chooses this proper prior to be informative instead of vague. Thus, it is crucial to perform sensitivity analysis whenever this prior is chosen to be informative.  

\section{Parameter Estimation for $\{\bfbeta, \bfgamma, \bfsigma^2 \}$} \label{app: parameter estimation}
Let $\bfb_1=\bfbeta_1$, $\bfb_t=(\bfbeta^\top_t, \gamma_{t-1})^\top$ for $t>1$, and $\bfb=(\bfb_1^\top, \ldots, \bfb_s)^\top$. The posterior distribution of $\bfb$ given $\bfy$ and $\hat{\bfphi}$ with objective priors $\pi(\bfb, \bfsigma^2) \propto \prod_{t=1}^s \sigma^{-2}_{t}$ is 
\begin{align*}
    \pi(\bfb \mid \bfy, \hat{\bfphi}) & \propto \int \sigma^{-2}_1 \pi(\bfy_1\mid \bfb_1, \sigma^2_1, \hat{\bfphi}_1) \prod_{t=2}^s \pi(\bfy_t\mid \bfy_{t-1}, \bfb_t, \sigma^2_t, \hat{\bfphi}_t) \sigma^{-2}_t \, d(\prod_{t=1}^s\sigma^2_t) \\
    & \propto \prod_{t}^s \{(\bfy_t - \bfX_t\bfb_t)^\top \bfR_t^{-1}(\bfy_t - \bfX_t\bfb_t)\}^{-n_t/2} |\bfR_t|^{-1/2}. 
\end{align*}
Maximization with respect to this posterior distribution yields that 
$$\hat{\bfb}_t = (\bfX_t^\top \bfR_t^{-1} \bfX_t)^{-1} \bfX_t^\top \bfR_t^{-1} \bfy_t.$$

Similarly, the posterior distribution of $\bfsigma^2$ given $\bfy$ and $\hat{\bfphi}$ with objective priors $\pi(\bfb, \bfsigma^2) \propto \prod_{t=1}^s \sigma^{-2}_{t}$ is 
\begin{align*}
    \pi(\bfsigma^2 \mid \bfy, \hat{\bfphi}) & \propto \int \sigma^{-2}_1 \pi(\bfy_1\mid \bfb_1, \sigma^2_1, \hat{\bfphi}_1) \prod_{t=2}^s \pi(\bfy_t\mid \bfy_{t-1}, \bfb_t, \sigma^2_t, \hat{\bfphi}_t) \sigma^{-2}_t \, d(\prod_{t=1}^s\bfb_t) \\
    & \propto \prod_{t=1}^s (\sigma_t^2)^{-(n_t-q_t)/2-1} |\bfR_t|^{-1/2} |\bfX_t^\top \bfR_t^{-1} \bfX_t|^{-1/2} \exp\{-S^2(\hat{\bfphi}_t)\}. 
\end{align*}
 It is easy to recognize that $\pi(\sigma_t^2\mid \bfy_{t-1}, \bfy_t, \bfphi_t) = \mathcal{IG}((n_t-q_t)/2, S^2(\hat{\bfphi}_t)/2)$. Hence, maximizing this posterior distribution with respect to $\sigma^2_t$ yields that 
 \begin{align*}
     \hat{\sigma}^2_t = S^2(\hat{\bfphi}_t) / (n_t-q_t+2).
 \end{align*}

\section{Input Variables in the Borehole Function} \label{app: testing function}
The input variables and their ranges in the Borehole function are given in Table~\ref{table: input in borehole}. 
\begin{table}[htbp]
\centering
\normalsize
   \caption{Input variables and their ranges in the Borehole function.}
  {\resizebox{1.0\textwidth}{!}{%
  \setlength{\tabcolsep}{3.0em}
   \begin{tabular}{l l} 
   \toprule 
input  & physical meaning \\
\midrule
$r_w\in [0.05, 0.15]$ & radius of borehole (m) \\ \noalign{\vskip 1.5pt}
$r\in[100, 50000]$  & radius of influence (m) \\ \noalign{\vskip 1.5pt}
$T_u\in [63070, 115600]$ & transmissivity of upper aquifer ($m^2$/yr)\\ \noalign{\vskip 1.5pt}
$H_u\in[990, 1110]$ & potentiometric head of upper aquifer (m) \\ \noalign{\vskip 1.5pt}
$T_{\ell} \in [63.1, 116]$ & transmissivity of lower aquifer ($m^2$/yr) \\ \noalign{\vskip 1.5pt}
$H_{\ell} \in [700, 820]$ & potentiometric head of lower aquifer (m) \\ \noalign{\vskip 1.5pt}
$L\in [1120, 1680]$ & length of borehold (m) \\ \noalign{\vskip 1.5pt}
$K_w\in[9855, 12045]$ & hygraulic conductivity of borehole (m/yr) \\
 \bottomrule
   \end{tabular}%
   }}
   \label{table: input in borehole}
\end{table}  

\section{Numerical Illustrations with the Mat\'ern-5/2 Correlation} \label{app: Result with Matern}
It is worth noting that when the Mat\'ern-5/2 correlation is used, the assumption in part (ii) of Assumption~\ref{ass: correlation} could be problematic due to the singularity of matrix $\mathbf{D}_{t, \ell}$. Thus the formulation of objective priors might be wrong. However, we still report the results using the previous formulation of objective priors in this situation. 

Table~\ref{table: borehole example} shows the results for the testbed with the Borehole function in Section~\ref{sec: borehole}. The conclusions under the Mat\'ern-5/2 correlation are very similar to those under the power-exponential correlation with roughness parameter 1.9. The only difference is that the empirical coverage probabilities are less than the nominal level 0.95 for objective priors under the power-exponential correlation function. This result makes sense since the power-exponential correlation with roughness parameter 1.9 yields process realizations that are not differentiable and hence the corresponding prediction will have larger uncertainty than those under process realizations that are twice differentiable due to the Mat\'ern-5/2 correlation.
\begin{table}[htbp]
\centering
\normalsize
   \caption{Predictive performance at 20 held-out inputs in the autoregressive cokriging model using the proposed objective priors and using the plug-in MLE approach in Gratiet \cite{Gratiet2013} when the Mat\'ern-5/2 correlation is used.}
     {\resizebox{1.0\textwidth}{!}{%
  \setlength{\tabcolsep}{2.0em}
   \begin{tabular}{l c c c } 
   \toprule 
   \noalign{\vskip 1.5pt}
& RMSE & CVG(95\%)   & ALCI(95\%)   \\  \noalign{\vskip 1.5pt}
\noalign{\vskip 1.5pt} \hline \noalign{\vskip 3pt}   \noalign{\vskip 1.5pt}
				 {Independent reference prior} &0.463      & 0.85 & 1.353  \\ 
				 \noalign{\vskip 4pt}				    
				 {Independent Jeffreys prior} &0.466   & 0.85   & 1.359  \\
				 \noalign{\vskip 4pt}  
		       {Jointly robust prior} & 0.379   & 0.95    &  1.436  \\
		       \noalign{\vskip 4pt} 
		       {plug-in MLE} &  1.940   & 1.00    & 17.85  \\
\noalign{\vskip 1.5pt} \bottomrule
   \end{tabular}%
   }}
   \label{table: borehole example}
\end{table}

Table~\ref{table: two-level example} and Table~\ref{table: three-level example} show the results for the 2-level cokriging model  and the 3-level cokriging model for the fluidized-bed processes in Section~\ref{sec: FBed}. The conclusions under the Mat\'ern-5/2 correlation are the same as those under the power-exponential correlation with roughness parameter 1.9.

\begin{table}[htbp]
\centering
\normalsize
   \caption{Predictive performance under 2-level cokriging based on $T_{2}$ and $T_{exp}$ using the proposed objective priors and the plug-in MLE approach in \cite{Gratiet2013} when the Mat\'ern-5/2 correlation is used.}
     {\resizebox{1.0\textwidth}{!}{%
  \setlength{\tabcolsep}{2.0em}
   \begin{tabular}{l c c c } 
   \toprule 
   \noalign{\vskip 1.5pt}
& RMSE & CVG(95\%)   & ALCI(95\%)   \\  \noalign{\vskip 1.5pt}
\noalign{\vskip 1.5pt} \hline \noalign{\vskip 3pt}   \noalign{\vskip 1.5pt}
{Independent reference prior} &0.513      & 0.88 & 2.357  \\ 
				 \noalign{\vskip 4pt}				    
{Independent Jeffreys prior} &0.563   & 0.88   & 1.944  \\
				 \noalign{\vskip 4pt}  
{Jointly robust prior} & 0.524   & 0.88    &  2.214  \\
		       \noalign{\vskip 4pt} 
 {plug-in MLE} &  2.219  & 0.88    & 7.315  \\
\noalign{\vskip 1.5pt} \bottomrule
   \end{tabular}%
   }}
   \label{table: two-level example}
\end{table}

\begin{table}[htbp]
\centering
\normalsize
   \caption{Predictive performance under 3-level cokriging based on $T_{1}$, $T_3$, and $T_{exp}$ using the proposed objective priors and the plug-in MLE approach in \cite{Gratiet2013} when the Mat\'ern-5/2 correlation is used.}
     {\resizebox{1.0\textwidth}{!}{%
  \setlength{\tabcolsep}{2.0em}
   \begin{tabular}{l c c c } 
   \toprule 
   \noalign{\vskip 1.5pt}
& RMSE & CVG(95\%)   & ALCI(95\%)   \\  \noalign{\vskip 1.5pt}
\noalign{\vskip 1.5pt} \hline \noalign{\vskip 3pt}   \noalign{\vskip 1.5pt}
{Independent reference prior} &1.309      & 1.00 & 6.591  \\ 
				 \noalign{\vskip 4pt}				    
{Independent Jeffreys prior} &1.222   & 1.00   & 6.611  \\
				 \noalign{\vskip 4pt}  
{Jointly robust prior} & 0.831  & 1.00    &  5.892  \\
		       \noalign{\vskip 4pt} 
{plug-in MLE} &  8.439  & 0.78    & 13.42  \\
\noalign{\vskip 1.5pt} \bottomrule
   \end{tabular}%
   }}
   \label{table: three-level example}
\end{table}

All these numerical results suggest that Bayesian analysis with the objective priors and the jointly robust priors are preferred than the plug-in MLE approach in \cite{Gratiet2013}. In addition, the reference prior should also be preferred over the Jeffreys prior in such Bayesian analysis. It is worth mentioning that different covariance models could have noticeable impact on the predictive performance. When the covariance model is assumed with more confidence (such as the Mat\'ern-5/2 correlation versus the power-exponential with roughness parameter 1.9), in general the predictive uncertainty tends to be smaller. However, obtaining a good estimator for the smoothness parameter that controls the differentiability of Gaussian process realizations remains an unresolved issue.

\section*{Acknowledgments}
The author is grateful to Professor James O. Berger for his insightful comments and suggestions. The author would like to thank an associate editor and two anonymous referees for their comments that significantly improve the presentation of the paper. 
\bibliographystyle{siamplain}
\bibliography{references}

\begin{thebibliography}{10}

\bibitem{Berger2006}
{\sc J.~O. Berger}, {\em {The case for objective Bayesian analysis}}, Bayesian
  Anal., 1 (2006), pp.~385--402, \url{https://doi.org/10.1214/06-BA115}.

\bibitem{Berger2015}
{\sc J.~O. Berger, J.~M. Bernardo, and D.~Sun}, {\em Overall objective priors},
  Bayesian Analysis, 10 (2015), pp.~189--221.

\bibitem{Berger2001}
{\sc J.~O. Berger, V.~De~Oliveira, and B.~Sanso}, {\em {Objective Bayesian
  analysis of spatially correlated data}}, J. Amer. Statist. Assoc., 96 (2001),
  pp.~1361--1374.

\bibitem{Cressie1993}
{\sc N.~Cressie}, {\em Statistics for Spatial Data}, John Wiley \& Sons, New
  York, revised~ed., 1993.

\bibitem{Dewettinck1999}
{\sc K.~Dewettinck, A.~D. Visscher, L.~Deroo, and A.~Huyghbbaet}, {\em Modeling
  the steady-state thermodynamic operation point of top-spray fluidized bed
  processing}, Journal of Food Engineering, 39 (1999), pp.~131--143.

\bibitem{DiceDesign}
{\sc D.~Dupuy, C.~Helbert, and J.~Franco}, {\em Dicedesign and diceeval: Two
  \textsf{R} packages for design and analysis of computer experiments}, Journal
  of Statistical Software, Articles, 65 (2015), pp.~1--38,
  \url{https://doi.org/10.18637/jss.v065.i11}.

\bibitem{Gu2019}
{\sc M.~Gu}, {\em Jointly robust prior for {G}aussian stochastic process in
  emulation, calibration and variable selection}, Bayesian Anal., 14 (2019),
  pp.~877--905.

\bibitem{Gu2018}
{\sc M.~Gu, X.~Wang, and J.~O. Berger}, {\em {Robust Gaussian stochastic
  process emulation}}, The Annals of Statistics, 46 (2018), pp.~3038--3066.

\bibitem{Harper1983}
{\sc W.~V. Harper and S.~K. Gupta}, {\em Sensitivity/uncertainty analysis of a
  borehole scenario comparing {Latin} hypercube sampling and deterministic
  sensitivity approaches}, tech. report, Battelle Memorial Inst., 1983.

\bibitem{Harville1974}
{\sc D.~A. Harville}, {\em Bayesian inference for variance components using
  only error contrasts}, Biometrika, 61 (1974), pp.~383--385,
  \url{http://www.jstor.org/stable/2334370}.

\bibitem{Kennedy2000}
{\sc M.~Kennedy and A.~O'Hagan}, {\em {Predicting the output from a complex
  computer code when fast approximations are available}}, Biometrika, 87
  (2000), pp.~1--13.

\bibitem{MuFiCokriging}
{\sc L.~{Le Gratiet}}, {\em MuFiCokriging: Multi-Fidelity Cokriging models},
  2012, \url{https://CRAN.R-project.org/package=MuFiCokriging}.
\newblock R package version 1.2.

\bibitem{Gratiet2013}
{\sc L.~Le~Gratiet}, {\em Bayesian analysis of hierarchical multifidelity
  codes}, SIAM/ASA J. on Uncertain. Quantif., 1 (2013), pp.~244--269.

\bibitem{Gratiet2014}
{\sc L.~Le~Gratiet and J.~Garnier}, {\em Recursive co-kriging model for design
  of computer experiments with multiple levels of fidelity}, Int. J. Uncertain.
  Quantif., 4 (2014), pp.~365--386.

\bibitem{ARCokrig}
{\sc P.~Ma}, {\em {ARCokrig}: Autoregressive cokriging models for multifidelity
  codes}, 2020, \url{https://CRAN.R-project.org/package=ARCokrig}.
\newblock R package version 0.1.0.

\bibitem{Ma2019Cov}
{\sc P.~Ma and A.~Bhadra}, {\em Kriging: {B}eyond {M}at\'ern}.
\newblock arXiv preprint arXiv:1911.05865, 2019.
\newblock https://arxiv.org/abs/1911.05865.

\bibitem{Ma2019PPCokriging}
{\sc P.~Ma, G.~Karagiannis, B.~A. Konomi, T.~G. Asher, G.~R. Toro, and A.~T.
  Cox}, {\em Multifidelity computer model emulation with high-dimensional
  output: An application to storm surge}, arXiv preprint arXiv:1909.01836,
  (2019).
\newblock https://arxiv.org/abs/1909.01836.

\bibitem{Morris1993}
{\sc M.~D. Morris, T.~J. Mitchell, and D.~Ylvisaker}, {\em Bayesian design and
  analysis of computer experiments: Use of derivatives in surface prediction},
  Technometrics, 35 (1993), p.~243.

\bibitem{Mure2018}
{\sc J.~Mur{\'e}}, {\em Propriety of the reference posterior distribution in
  {Gaussian} process regression}, arXiv preprint arXiv:1805.08992,  (2018).
\newblock https://arxiv.org/abs/1805.08992.

\bibitem{Nocedal2006}
{\sc J.~Nocedal and S.~J. Wright}, {\em Numerical Optimization}, Springer, New
  York, NY, USA, second~ed., 2006.

\bibitem{OHagan2006}
{\sc A.~O'Hagan}, {\em {Bayesian analysis of computer code outputs: A
  tutorial}}, Reliability Engineering {\&} System Safety, 91 (2006),
  pp.~1290--1300.

\bibitem{Paulo2005}
{\sc R.~Paulo}, {\em {Default priors for Gaussian processes}}, The Annals of
  Statistics, 33 (2005), pp.~556--582.

\bibitem{Peherstorfer2018}
{\sc B.~Peherstorfer, K.~Willcox, and M.~Gunzburger}, {\em Survey of
  multifidelity methods in uncertainty propagation, inference, and
  optimization}, SIAM Review, 60 (2018), pp.~550--591.

\bibitem{Qian2008}
{\sc P.~Z.~G. Qian and C.~F.~J. Wu}, {\em {Bayesian} hierarchical modeling for
  integrating low-accuracy and high-accuracy experiments}, Technometrics,
  (2008).

\bibitem{Sacks1989}
{\sc J.~Sacks, W.~J. Welch, T.~J. Mitchell, and H.~P. Wynn}, {\em Design and
  analysis of computer experiments}, Statist. Sci., 4 (1989), pp.~409--435.

\bibitem{Santner2018}
{\sc T.~J. Santner, B.~J. Williams, and W.~I. Notz}, {\em {The design and
  analysis of computer experiments; 2nd ed.}}, Springer series in statistics,
  Springer, New York, 2018.

\bibitem{Xiong2013}
{\sc S.~Xiong, P.~Z.~G. Qian, and C.~F.~J. Wu}, {\em Sequential design and
  analysis of high-accuracy and low-accuracy computer codes}, Technometrics, 55
  (2013), pp.~37--46.

\end{thebibliography}

\end{document}